\documentclass[a4paper,11pt,pdf]{amsart}
\usepackage{enumerate, amsmath, amsfonts, amssymb, amsthm, thmtools, wasysym, graphics, graphicx, xcolor, frcursive,comment,bbm}

\usepackage{qtree}
\usepackage{etex}
\usepackage{lscape}
\usepackage{tikz}
\usepackage{tikz-cd}
\usetikzlibrary{arrows,shapes,chains,matrix,positioning,scopes,cd,patterns}
\usetikzlibrary{decorations.pathmorphing}
\usetikzlibrary{decorations.pathreplacing,calligraphy}
\usetikzlibrary{backgrounds,calc,positioning}
\usetikzlibrary{arrows.meta}
\usetikzlibrary{shapes.geometric,calc}
\usepackage[all]{xy}
\usepackage{forest}
\usepackage{rotating}
\newcommand\leqdot{\mathrel{\ooalign{$\leq$\cr
  \hidewidth\raise0.225ex\hbox{$\cdot\mkern0.5mu$}\cr}}}

\makeatletter
\newcommand{\thickhline}{%
    \noalign {\ifnum 0=`}\fi \hrule height 1pt
    \futurelet \reserved@a \@xhline
}

\definecolor{darkblue}{rgb}{0.0,0,0.7} 
\definecolor{darkred}{rgb}{0.7,0,0} 
\usepackage{hyperref}
\usepackage{cleveref}
\usepackage[all]{xy}
\usepackage[T1]{fontenc}

\usepackage{vmargin}            
\setmarginsrb{3cm}{3.2cm}{3cm}{3.1cm}{0cm}{0.6cm}{0cm}{0cm}

\usepackage{caption,lipsum}
\usepackage{graphicx}                  
\usepackage{pstricks,pst-plot,pst-text,pst-tree,pst-eps,pst-fill,pst-node,pst-math}
\usepackage{setspace}
\usepackage{multicol}

\newcommand{\darkred}{\color{darkred}} 
\newcommand{\defn}[1]{\emph{\darkred #1}} 

\def\Div{{\sf{Div}}}

\usepackage{etex}

\newtheorem{theorem}{Theorem}[section]
\newtheorem{prop}[theorem]{Proposition}
\newtheorem{lemma}[theorem]{Lemma}
\newtheorem{cor}[theorem]{Corollary}

\newtheorem*{question}{Question}

\theoremstyle{definition}
\newtheorem{definition}[theorem]{Definition}
\newtheorem{rmq}[theorem]{Remark}
\newtheorem{exple}[theorem]{Example}

\numberwithin{equation}{section}

\def\Div{{\mathsf{Div}}}

\crefname{figure}{figure}{figures}
\Crefname{figure}{Figure}{Figures}

\title[Odd and even Fibonacci lattices arising from a Garside monoid]{Odd and even Fibonacci lattices arising from a Garside monoid}

\author{Thomas Gobet}
\address{Institut Denis Poisson, CNRS UMR 7350, Faculté des Sciences et Techniques, Université de Tours, Parc de Grandmont, 
37200 TOURS, France}
\email{thomas.gobet@lmpt.univ-tours.fr}

\author{Baptiste Rognerud}
\address{Institut de Mathématiques de Jussieu, Paris Rive Gauche (IMJ-PRG), Campus des Grands Moulins,
Université de Paris - Boite Courrier 7012, 8 Place Aurélie Nemours,
75205 PARIS Cedex 13, France}
\email{baptiste.rognerud@imj-prg.fr}

\begin{document}
\maketitle

\begin{abstract}
We study two families of lattices whose number of elements are given by the numbers in even (respectively odd) positions in the Fibonacci sequence. The even Fibonacci lattice arises as the lattice of simple elements of a Garside monoid partially ordered by left-divisibility, and the odd Fibonacci lattice is an order ideal in the even one. We give a combinatorial proof of the lattice property, relying on a description of words for the Garside element in terms of Schröder trees, and on a recursive description of the even Fibonacci lattice. This yields an explicit formula to calculate meets and joins in the lattice. As a byproduct we also obtain that the number of words for the Garside element is given by a little Schröder number. 
\end{abstract}

\tableofcontents

\thispagestyle{empty}

\section{Introduction}

Several algebraic structures naturally yield examples of lattices: as elementary examples, one can cite the lattice of subsets of a given set ordered by inclusion, or the lattice of subgroups of a given group. One can then study which properties are satisfied by the obtained lattices, or conversely, starting from a known lattice, wondering for instance if it can be realized in a given algebraic framework, or if a property of the lattice implies properties of the attached algebraic structure(s) and vice-versa. 

The aim of this paper is to give a combinatorial description of a finite lattice that appeared in the framework of \textit{Garside theory}. We will not recall results and principles of Garside theory as they will not be used in this paper, but the interested reader can look at~\cite{DP, Garside} for more on the topic. This is a branch of combinatorial group theory which aims at establishing properties of families of infinite groups such as the solvability of the word problem, the conjugacy problem, the structure of the center, etc. Roughly speaking, a Garside group is a group of fraction of a monoid (called a \textit{Garside monoid}) with particularly nice divisibility properties, which ensures that the above-mentioned problems can be solved. Such a monoid $M$ has no nontrivial invertible element, and comes equipped with a distinguished element $\Delta$ (called a~\textit{Garside element}) whose left- and right-divisors are finite, coincide, generate the monoid, and form a lattice under left- and right-divisibility. The left- or right-divisors of $\Delta$ are called the~\textit{simples}. 

The fundamental example of a Garside group is the $n$-strand Artin braid group~\cite{garside69}. It admits several non-equivalent Garside structures (i.e., nonisomorphic Garside monoids whose group of fractions are isomorphic to the $n$-strand braid group), and the lattice of simples in the first discovered such Garside structure is isomorphic to the weak Bruhat order on the symmetric group. Several widely studied lattices can be realized as lattices of simples of a Garside monoid: this includes the lattices of left and right weak Bruhat order on any finite Coxeter group~\cite{BS, Deligne}, the lattice of (generalized) noncrossing partitions attached to a finite Coxeter group~\cite{Dual, BKL}, etc. (see also~\cite{Picantin} for many other examples). This suggests the following question: 

\begin{question}
Which lattices can appear as lattices of simples of Garside monoids ?
\end{question}

The aim of this paper is to study a family $P_n$ of lattices arising as simples of a family $M_n$, $n\geq 2$ of Garside monoids introduced by the first author~\cite{Gobet}. For $n=2$, the corresponding Garside group is isomorphic to the $3$-strand braid group $B_3$, while in general it is isomorphic to the $(n,n+1)$-torus knot group, which for $n>3$ is a (strict) extension of the $(n+1)$-strand braid group $B_{n+1}$. The lattice property of $P_n$ follows from the fact proven in~\emph{op.\ cit.} that $M_n$ is a Garside monoid, but it gives very little information about the structure and properties of the lattice. For instance, one does not have a formula enumerating the number of simples, and only an algorithm to calculate meet and joins in the lattice. 

In Section~\ref{sec_alg} we give a new proof of the lattice property of $P_n$ (Theorem~\ref{thm_lattice_main}) by exhibiting the recursive structure of the poset. Every lattice $P_n$ turns out to contain the lattices $P_i$, $i< n$ as sublattices. Note that an ingredient of the proof of Theorem~\ref{thm_lattice_main} is proven later on in the paper, as it relies on a combinatorial description for the set of words for the Garside element in terms of Schröder trees. More precisely, in Section~\ref{sec_Schröder} we establish a simple bijection between the set of words for $\Delta_n$ and the set of Schröder trees on $n+1$ leaves, in such a way that applying a defining relation of $M_n$ to a word amounts to applying what we call a "local move" on the corresponding Schröder tree (Theorem~\ref{thm_main_reduced} and Corollary~\ref{cor_graph}). These local moves are given by specific edge contraction and are related to the notion of refinement considered in~\cite{loday}. This allows us to establish in Proposition~\ref{dni_trees} an isomorphism of posets between subposets of $P_n$ and $P_i$, $i< n$, required in the proof of Theorem~\ref{thm_lattice_main}. 

Finally, the obtained recursive description of $P_n$ together with the description of words for $\Delta_n$ in terms of Schröder trees allows us to derive a few enumerative results. This is done in Section~\ref{sec_enumerative}. The first one is that the number of elements of $P_n$ is given by $F_{2n}$, where $F_i$ is the $i$-th Fibonacci number (Lemma~\ref{cor_fib_even}). We thus call $P_n$ the~\textit{even Fibonacci lattice}. The atoms of $M_n$ turn out to have the same left- and right-lcm, which is strictly less than $\Delta_n$. We also show that the sublattice of $P_n$ defined as the order ideal of this lcm has $F_{2n-1}$ elements (Lemma~\ref{lem_odd_fib}), and thus call it the~\textit{odd Fibonacci lattice}. Other enumerative results include the determination of the number of words for the Garside elements (Corollary~\ref{cor_words_g}), and the number of words for the whole set of simples (Theorem~\ref{thm_words_whole}). 

Recall that the Garside monoid $M_n$ under study in this paper has group of fractions isomorphic to the $(n,n+1)$-torus knot group. This Garside structure was generalized to all torus knot groups in~\cite{gobet_bis}. It would be interesting to have a description of the lattices of simples of this bigger family of Garside monoids. 
  
\begin{figure}
 	
	\begin{center}
 		\begin{pspicture}(0,0)(12,13)
 			\rput(4,0){\color{blue}\tiny $1$}
 			\rput(4,1){\color{blue}\tiny $\rho_1$}
 			\rput(2,2){\color{blue}\tiny $\rho_2$}
 			\rput(0,3){\color{blue}\tiny $\rho_2 \rho_1$}
 			\rput(8,3){\color{blue}\tiny $\rho_3$}
 			\rput(4,4){\color{blue}\tiny $\rho_1 \rho_3$}
 			\rput(10,4){\color{blue}\tiny $\rho_3 \rho_1$}
 			\rput(2,5){\color{blue}\tiny $\rho_1 \rho_3 \rho_1$}
 			\rput(4.3,5.2){\tiny $\rho_3 \rho_2$}
 			\rput(0,6){\color{blue}\tiny $\rho_2 \rho_1 \rho_3$}
 			\rput(5,6.5){\tiny $\rho_3 \rho_2 \rho_1$}
 			\rput(7.7,6){\color{blue}\tiny $\rho_3^2$}
 			\rput(10,7){\color{blue}\tiny $\rho_3 \rho_1 \rho_3$}
 			\rput(12,7){\tiny $\rho_3^2 \rho_1$}
 			\rput(2,8){\color{blue}\tiny $(\rho_1 \rho_3)^2$}
 			\rput(9.8,8){\tiny $(\rho_3 \rho_1)^2$}
 			\rput(4,9){\color{blue}\tiny $\rho_3^3$}
 			\rput(8,9){\tiny $\rho_3 \rho_2 \rho_1 \rho_3$}
 			\rput(12,10){\tiny $\rho_3^2 \rho_1 \rho_3$}
 			\rput(8,11){\tiny $(\rho_3 \rho_1)^2 \rho_3$}
 			\rput(6,12){\tiny $\rho_3^4$}
 			\psline[linecolor=blue](4,0.25)(4,0.75)
 			\psline[linecolor=blue](4,1.25)(4,3.75)
 			\psline[linecolor=blue](3.8,0.25)(2.25,2)
 			\psline[linecolor=blue](4.2,0.25)(7.8,2.9)
 			\psline[linecolor=blue](8.2,3.2)(9.6,3.8)
 			\psline[linecolor=blue](2,2.25)(2,4.75)
 			\psline[linecolor=blue](0,3.25)(0,5.75)
 			\psline[linecolor=blue](1.8,2.1)(0.4,3)
 			\psline[linecolor=blue](4.4,4.2)(7.4,5.8)
 			\psline[linecolor=blue](3.6,4.2)(2.55,5)
 			\psline[linecolor=blue](2,5.25)(2,7.75)
 			\psline(7.92,3.25)(4.3,5)
 			\psline[linecolor=blue](10,4.25)(10,6.75)
 			\psline(10,7.25)(10,7.75)
 			\psline(7.9,6)(11.6,7)
 			\psline(12,7.25)(12,9.75)
 			\psline[linecolor=blue](9.4,7)(4.3,9)
 			\psline[linecolor=blue](0,6.2)(1.7,7.75)
 			\psline[linecolor=blue](2.6,8.2)(3.8,8.9)
 			\psline(4.2,9.1)(5.8,11.9)
 			\psline(5,6.65)(8,8.75)
 			\psline(8,9.25)(8,10.75)
 			\psline(10,8.25)(8.2,10.75)
 			\psline[linecolor=blue](7.5,6.1)(4,8.75)
 			\psline(11.6,10.3)(6.3,12)
 			\psline(7.85,11.18)(6.3,11.9)
 			\psline[linecolor=blue](8.1,3.25)(7.7,5.75)
 			\psline(4.5,5.4)(9.5,7.75)
 			\psline(4.3,5.45)(4.8,6.3)
 			
 		\end{pspicture}
 	\end{center} 
 	\caption{The even Fibonacci lattice for $n=3$ and (in blue) the odd Fibonacci lattice inside it.}
 	\label{simples_3}
 \end{figure}

\section{Definition and structure of the poset}\label{sec_alg}

\subsection{Definition of the poset}

The beginning of this section is devoted to explaining how the poset under study is defined. We recall the definition of the monoid from which it is built, as well as a few properties of this monoid (all of which are proven in~\cite{Gobet}). 

Let $M$ be a monoid and $a,b\in M$. We say that $a$ is a \defn{left divisor} of $b$ (or that $b$ is a \defn{right multiple} of $a$) if there is $c\in M$ such that $ac=b$. We similarly define right divisors and left multiples. 

Let $M_0$ be the trivial monoid and for $n\geq 1$, let $M_n$ be the monoid defined by the presentation \begin{equation}\label{old}
\bigg\langle \rho_1, \rho_2, \dots, \rho_n \ \bigg\vert\ \rho_1  \rho_n \rho_{i} =\rho_{i+1} \rho_n~\text{for all }1 \leq i \leq n-1
\ \bigg\rangle.
\end{equation} We denote by $\mathcal{S}$ the set of generators $\{\rho_1, \rho_2, \dots, \rho_n\}$, and by $\mathcal{R}$ the defining relations of $M_n$. This monoid was introduced by the first author in~\cite[Definition 4.1]{Gobet}. Note that this monoid is equipped with a length function $\lambda : M_n\longrightarrow \mathbb{Z}_{\geq 0}$ given by the multiplicative extension of $\lambda(\rho_i)=i$ for all $i=1, \dots, n$, which is possible since the defining relations do not change the length of a word. As a corollary, the only invertible element in $M_n$ is the identity, and the left- and right-divisibility relations are partial orders on $M_n$. We will write $a\leq_L b$ or simply $a\leq b$ if $a$ left-divides $b$, and $a\leq_R b$ if $a$ right-divides $b$.  

This monoid was shown to be a so-called \defn{Garside monoid} (see~\cite[Theorem 4.18]{Gobet}), with corresponding Garside group (which has the same presentation as $M_n$) isomorphic to the $(n,n+1)$-torus knot group, that is, the fundamental group of the complement of the torus knot $T_{n,n+1}$ in $S^3$. Garside monoids have several important properties. Among them, the left- and right-divisibility relations equip $M_n$ with two lattice structures, and $M_n$ comes equipped with a distinguished element $\Delta_n$, called a~\defn{Garside element}, which has the following two properties
\begin{enumerate}
\item The set of left divisors of $\Delta_n$ coincides with its set of right divisors, and forms a finite set.
\item The set of left (or right) divisors of $\Delta_n$ generates $M_n$. 
\end{enumerate}
This Garside element is given by $\Delta_n=\rho_n^{n+1}$. In particular, as any Garside monoid is a lattice for both left- and right-divisibility, the set $\mathsf{Div}(\Delta_n)$ of left (or right) divisors of $\Delta_n$ is a finite lattice if equipped by the order relation given by the restriction of left- (or right-) divisibility on $M_n$. The set $\mathsf{Div}(\Delta_n)$ is the set of \defn{simple elements} or \defn{simples} of $M_n$. In general $(\mathsf{Div}(\Delta_n), \leq_L)$ and $(\mathsf{Div}(\Delta_n), \leq_R)$ will not be isomorphic as posets. But we always have $$(\mathsf{Div}(\Delta_n), \leq_L)\cong(\mathsf{Div}(\Delta_n), \leq_R)^{\mathrm{op}}$$ (see for instance~\cite[Lemma 2.19]{Gobet}; such a property holds in any Garside monoid). 

We will give a new proof that $(\mathsf{Div}(\Delta_n), \leq)$ (and hence $(\mathsf{Div}(\Delta_n), \leq_R)$ is a lattice, in a way which will exhibit a recursive structure of the poset. To this end, we will require (sometimes without mentioning it) a few basic results on the monoid $M_n$ which are either explained above or proven in~\cite{Gobet}:
\begin{enumerate}
\item The left- and right-divisibility relations on $M_n$ are partial orders. 
\item The monoid $M_n$ is both left- and right-cancellative, i.e., for $a, b, c\in M_n$, we have that $ab=ac\Rightarrow b=c$, and $ba=ca\Rightarrow b=c$ (see~\cite[Propositions 4.9 and 4.12]{Gobet}), 
\item The set of left- and right-divisors of $\Delta_n$ coincide. In fact, the element $\Delta_n$ is central in $M_n$, hence as $M_n$ is cancellative, for $a,b\in M_n$ such that $ab=\Delta_n$, we have $ab=ba$ (see~\cite[Proposition 4.15]{Gobet}) 
\end{enumerate}

\subsection{Lattice property}

The aim of this subsection is to prove a few properties of simple elements of $M_n$, and to derive a new algebraic proof that $\mathsf{Div}(\Delta_n)$ is a lattice. 

\begin{prop}\label{prop_decomposition}
Let $x_1 x_2\cdots x_k$ be a word for $\Delta_n$, with $x_i\in\mathcal{S}$ for all $i=1, \dots, k$. There are $i_1=1 < i_2 < \dots < i_\ell \leq k$ such that \begin{itemize}
\item For all $j=1, \dots, \ell$, the word $y_j:=x_{i_j} x_{i_j+1} \cdots x_{i_{j+1}-1}$ (with the convention that $i_{\ell+1}=k+1$) is a word for a power of $\rho_n$,
\item The decomposition $y_1|y_2|\cdots|y_\ell$ of the word $x_1 x_2 \cdots x_k$ is maximal in the sense that no word among the $y_j$ can be decomposed as a product of two nonempty words which are words for powers of $\rho_n$. 
\end{itemize}
Morever, a decomposition with the above properties is unique. 
\end{prop}
\begin{proof}
The existence of the decomposition is clear using the fact that $M_n$ is cancellative: given the word $x_1 x_2 \cdots x_k$, consider the smallest $i\in\{1, 2, \dots, k\}$ such that $x_1 x_2 \cdots x_k$ is a word for a power of $\rho_n$. Such an $i$ has to exist, as $x_1 x_2\cdots x_k$ is a word for a power of $\rho_n$. Then set $i_2:=i+1$. By cancellativity in $M_n$, since $x_1 \cdots x_i$ and $x_1 \cdots x_k$ are both words for a power of $\rho_n$, the word $x_{i+1}\cdots x_k$ must also be a word for a power of $\rho_n$. Hence one can go on, arguing the same with the word $x_{i+1} \cdots x_k$. Again by cancellativity, this decomposition must be maximal.

Now assume that the decomposition is not unique, that is, assume that $y_1 | y_2| \cdots | y_\ell$ and $z_1 | z_2| \cdots | z_{\ell'}$ are two decompositions of the word $x_1 x_2\cdots x_k$ satisfying the properties of the statement. As both $y_1$ and $z_1$ are words for a power of $\rho_n$, if $y_1 \neq z_1$, then one word must be strict prefix of the other, say $z_1$ is a strict prefix of $y_1$. But this contradicts the maximality of the decomposition $y_1| y_2|\cdots | y_\ell$: indeed if $y_1= x_1 x_2 \cdots x_{i_2-1}$ and $z_1=x_1 x_2\cdots x_p$ with $p< i_2-1$, we can decompose $y_1$ nontrivially as $x_1 x_2\cdots x_p| x_{p+1} \cdots x_{{i_2}-1}$, and by cancellativity both $x_1 \cdots x_p$ and $x_{p+1} \cdots x_{i_2-1}$ are words for powers of $\rho_n$. 
\end{proof}

\begin{exple}
Consider the word $\rho_3 \rho_1 \rho_7 \rho_1 \rho_7 \rho_5 \rho_4 \rho_7 \rho_7 \rho_1 \rho_7 \rho_6$ in $M_7$. We claim that this is a word for the Garside element $\rho_7^8$ of $M_7$. Indeed, using the defining relation $\rho_1 \rho_7 \rho_i=\rho_{i+1} \rho_7$ with $i=5$ and $6$, we get that $$\rho_3 (\rho_1 \rho_7 \rho_1 \rho_7 \rho_5) \rho_4 \rho_7 \rho_7 (\rho_1 \rho_7 \rho_6)=\rho_3 \rho_7^3 \rho_4 \rho_7^4,$$ and we observe also applying defining relations that $$\rho_3 \rho_7^3 \rho_4 \rho_7^4=\rho_1 \rho_7 \rho_2 \rho_7^2 \rho_4 =\rho_1 \rho_7 \rho_1 \rho_7 \rho_1\rho_7 \rho_4=\rho_1 \rho_7 \rho_1 \rho_7 \rho_5\rho_7=\rho_1 \rho_7 \rho_6 \rho_7^2=\rho_7^4.$$

 The decomposition according to Proposition~\ref{prop_decomposition} is given by $$\underbrace{\rho_3 \rho_1 \rho_7 \rho_1 \rho_7 \rho_5 \rho_4 }_{:=y_1}|\underbrace{\rho_7}_{:=y_2}| \underbrace{\rho_7}_{:=y_3} |\underbrace{\rho_1 \rho_7 \rho_6}_{:=y_4}.$$ It is indeed clear by considering $\lambda(u)$ for $u$ prefixes of $y_1$ or $y_4$ that whenever $u$ is a proper prefix, we do not have $\lambda(u)$ equal to a multiple of $7$, which is a necessary condition for a word to represent a power of $\rho_7$. 
\end{exple}

\begin{lemma}\label{lem:atomes_gauche}
Let $1\leq k \leq n$. Then $$\mathcal{S}\cap \{x\in  \Div(\Delta_n) \ | \ x\leq \rho_k \rho_n^k\}=\{\rho_1, \rho_2, \dots, \rho_k\}.$$
\end{lemma}

\begin{proof}
We argue by induction on $k$. The result is clear for $k=1$, as no defining relation of $M_n$ can be applied to the word $\rho_1 \rho_n$. Now let $k>1$. Observe that $$\rho_k \rho_n^k = (\rho_1 \rho_n)^k= (\rho_1 \rho_n) (\rho_1 \rho_n)^{k-1}.$$ In particular we have $\rho_1\leq \rho_k \rho_n^k$ and by induction, we get $\rho_1 \rho_n \rho_i\leq \rho_k \rho_n^k$ for all $1\leq i \leq k-1$. As $\rho_1 \rho_n \rho_i=\rho_{i+1} \rho_n$ we get that $\{\rho_1, \rho_2, \dots, \rho_k\}\subseteq \mathcal{S}\cap \{x\in  \Div(\Delta_n) \ | \ x\leq \rho_k \rho_n^k\}$. 

It remains to show that no other $\rho_i$ can be a left-divisor of $\rho_k \rho_n^k$. Hence assume that $i>k$ and $\rho_i \leq \rho_k \rho_n^k$. Hence there is a word $x_1 x_2 \cdots x_p$ for $\rho_k \rho_n^k$, where $x_i\in\mathcal{S}$ for all $i$, such that $x_1=\rho_i$. As the words $x_1 x_2 \cdots x_p$ and $\rho_k \rho_n^k$ represent the same element, they can be related by a finite sequence of words $w_0=x_1 x_2 \cdots x_p, w_1, \dots, w_q=\rho_k \rho_n^k$, where each $w_i$ is a word with letters in $\mathcal{S}$ and $w_{i+1}$ is obtained from $w_i$ by applying a single relation somewhere in the word. As the first letter of $w_0$ differs from the first letter of $w_q$, there must exist some $0\leq \ell < q$ such that $w_\ell$ begins by $\rho_i$ but $w_{\ell+1}$ does not. It follows that the relation allowing one to pass from $w_\ell$ to $w_{\ell+1}$ has to be applied at the beginning of the word $w_\ell$. But the only possible relation with one side beginning by $\rho_i$ is $\rho_i \rho_n=\rho_1 \rho_n \rho_{i-1}$. It follows that $\rho_1 \rho_n \rho_{i-1}\leq   \rho_k \rho_n^k = (\rho_1 \rho_n)^k$. By cancellativity, we get that $\rho_{i-1}\leq (\rho_1 \rho_n)^{k-1}=\rho_{k-1} \rho_n^{k-1}$. By induction this forces one to have $i-1 \leq k-1$, contradicting our assumption that $i>k$.    \end{proof}

Similarly, we have

\begin{lemma}\label{lem:atomes_droite}
Let $1\leq k \leq n$. Then $$\mathcal{S}\cap \{x\in  \Div(\Delta_n) \ | \ x\leq_R \rho_n^k\}=\{\rho_n, \rho_{n-1}, \dots, \rho_{n-k+1}\}.$$
\end{lemma}

\begin{proof}
As for \Cref{lem:atomes_gauche}, we argue by induction on $k$. The result is clear for $k=1$. Hence assume that $k>1$. As $(\rho_1 \rho_n)^{n-j} \rho_j=\rho_n^{n-j+1}$, we get that $\rho_j \leq_R \rho_n^k$ for all $j$ such that $n-j+1\leq k$, that is, for all $j\geq n-k+1$. It remains to show that no other $\rho_j$ can right-divide $\rho_n^k$. Hence assume that $\rho_j\leq_R \rho_n^k$, where $j < n-k+1$. Arguing as in the proof of \Cref{lem:atomes_gauche}, we see that $\rho_1 \rho_n \rho_j = \rho_{j+1} \rho_n$ must be a right-divisor of $\rho_n^k$, hence by cancellativity that $\rho_{j+1}\leq_R \rho_n^{k-1}$. By induction this forces $j+1\geq n-k+2$, contradicting our assumtion that $j < n-k+1$.\end{proof} 

For $x\in\Div(\Delta_n)$, let $d(x):=\max\{ k\geq 0\ | \ \rho_n^k\leq x\}$. Let $0\leq i \leq n+1$ and let \[D_n^i:=\{x\in\Div(\Delta_n)~|~d(x)=i \}. \] Note that \[ \Div(\Delta_n)=\coprod_{0\leq i \leq n+1} D_n^i. \] We have $D_n^n=\{ \rho_n^n \}$, $D_n^{n+1}=\{ \Delta_n \}$.

\begin{lemma}\label{lemma_word}
Let $x\in \mathsf{Div}(\Delta_n)$ and $i=d(x)$. Let $x'\in M_n$ such that $x=\rho_n^i x'$. Note that $x'\in D_n^0$. Let $x_1 x_2 \cdots x_k$ be a word for $x$, where $x_i\in\mathcal{S}$ for all $i=1, \dots, k$. Then there is $1\leq \ell \leq k$ such that $x_1 x_2 \cdots x_\ell$ is a word for $\rho_n^i$ (and hence $x_{\ell+1} \cdots x_k$ is a word for $x'$ by cancellativity). In other words, any word for $x$ has a prefix which is a word for $\rho_n^i$.   
\end{lemma}

\begin{proof}
It suffices to show that if $z_1 z_2 \cdots z_p$ is an expression for $x$ such that $z_1 z_2 \cdots z_q$ is an expression for $\rho_n^i$ ($q\leq p$, then one cannot apply a defining relation of $M_n$ on the word $z_1 z_2 \cdots z_p$ simultaneously involving letters of the word $z_1 z_2 \cdots z_q$ and letters of the word $z_{q+1} \cdots z_p$. Let us consider the three possible cases where this could occur: one could have $\rho_1 \rho_n | \rho_j$, $\rho_1|\rho_n \rho_j$, or $\rho_{j+1}| \rho_n$ ($1\leq j < n$), where the $|$ separates the letters $z_q$ and $z_{q+1}$. The last two cases cannot happen, since one would have $z_{q+1}=\rho_n$, hence $z_{q+1} \cdots z_p$ would be a word for $x'$ beginning by $\rho_n$, contradicting the fact that $x'\in D_n^0$. It remains to show that the case $\rho_1 \rho_n|\rho_j$ cannot happen. Hence assume that $z_{q-1}=\rho_1, z_q=\rho_n, z_{q+1}=\rho_j$. By cancellativity, as $z_1 z_2\cdots z_q$ is a word for $\rho_n^i$, it implies that $\rho_1 \leq_R \rho_n^{i-1}$. By~\cref{lem:atomes_droite}, this implies that $n-(i-1)+1=1$, hence that $i=n+1$. Since $x\in \mathsf{Div}(\Delta_n)$ and $x=\rho_n^{n+1} x'=\Delta_n x'$, we get $x'=1$, contradicting the fact that $z_{q+1}=\rho_j$.     
\end{proof}

\begin{lemma}\label{rho_n_j_entre}
Let $i,j\in\{0, 1, \dots, n+1\}$, with $i\neq j$. Let $x\in D_n^i$, $y\in D_n^j$. Assume that $x\leq y$. Then $i<j$ and $x< \rho_n^j \leq y$.  
\end{lemma}

\begin{proof}
It is clear that $i<j$, since $\rho_n^i\leq y$ as $\rho_n^i\leq x$, hence $j < i$ would contradict $y\in D_n^j$. In particular $x < y$. Let $x',y'$ such that $x=\rho_n^i x'$ and $y=\rho_n^j y'$. Note that $x',y'$ both lie in $D_n^0$. Since $x\leq y$ and $M_n$ is cancellative, we get that $x'< \rho_n^{j-i} y'$. It implies that there exists a word $x_1 x_2 \cdots x_k$ for $\rho_n^{j-i} y'$ ($x_i\in\mathcal{S}$) and $1\leq \ell <  k$ such that $x_1 x_2\cdots x_\ell$ is a word for $x'$. Now by~\cref{lemma_word}, there is $0\leq \ell'\leq k$ such that $x_1 x_2\cdots x_{\ell'}$ is a word for $\rho_n^{j-i}$. If $\ell'\leq\ell$, then $\rho_n^{j-i}\leq x'$, contradicting the fact that $x'\in D_n^0$. Hence $\ell'>\ell$, and $x'<\rho_n^{j-i}$. Multiplying by $\rho_n^i$ on the left we get $x<\rho_n^j$.    
\end{proof}

\begin{lemma}\label{sorties}
Let $z_1, z_2\in D_n^i$. Let $1\leq k_1 < k_2 \leq n$ and assume that there are two cover relations $z_1 \leqdot \rho_n^{k_1}$, $z_2\leqdot \rho_n^{k_2}$ in $(\mathrm{Div}(\Delta_n), \leq)$. Then $z_1 < z_2$. 
\end{lemma}

\begin{proof}
As $z_1 \leq \rho_n^{k_1}$, $z_2\leq \rho_n^{k_2}$ are cover relations, there are $1\leq j_1, j_2\leq n$ such that $z_1 \rho_{j_1}=\rho_n^{k_1}$, $z_2 \rho_{j_2}=\rho_n^{k_2}$. By~\cref{lem:atomes_droite}, for $\ell\in\{1,2\}$ we have $j_\ell\in\{n-k_\ell+1, \dots, n\}$ and \[ \rho_n^{k_\ell}=\rho_n^{j_\ell+k_\ell-1-n}(\rho_1 \rho_n)^{n-j_\ell} \rho_{j_\ell}. \]
In particular, we have $$z_\ell=\rho_n^{j_\ell+k_\ell-1-n}(\rho_1 \rho_n)^{n-j_\ell}$$ and as $(\rho_1 \rho_n)^{n-j_\ell}=\rho_{n-j_\ell} \rho_n^{n-j_\ell}$, by~\cref{lem:atomes_gauche} we see that $\rho_n$ cannot be a left divisor of $(\rho_1 \rho_n)^{n-j_\ell}$, and hence that $d(z_\ell)=j_\ell+k_\ell-1-n$. But $d(z_\ell)=i$ for $\ell\in\{1,2\}$, and since $k_1 < k_2$ we deduce that $j_1 > j_2$. Since $z_\ell=\rho_n^i (\rho_1 \rho_n)^{n-j_\ell}$, we get that $z_1 < z_2$, which concludes the proof.  
\end{proof}

\begin{theorem}\label{thm_lattice_main}
The poset $(\Div(\Delta_n), \leq)$ is a lattice. Given $i,j\in\{0,1,\dots, n+1\}$ with $i\leq j$ and $x\in D_n^i$, $y\in D_n^j$, we have \[ x \wedge y=
        x \wedge_i \left(\bigvee_i\{z\in D_n^i \ | \ z\leq y\}\right), \] where $\vee_i$ and $\wedge_i$ denote the meet and join on the restriction of the left-divisibility order on $D_n^i$, which itself forms a lattice. Note that if $i=j$ we simply get $x\wedge y=x \wedge_i y$. 
\end{theorem}

\begin{proof}
The proof is by induction on $n$. We have $\mathsf{Div}(\Delta_0)=\{\bullet\}$, and $\mathsf{Div}(\Delta_1)=\{1, \rho_1, \rho_1^2\}$, which is a lattice. Hence assume that $n \geq 2$. By Proposition~\ref{dni_trees} below, the restriction of the left-divisibility to $D_n^i$ yields an isomorphism of poset with $\mathsf{Div}(\Delta_{n-i})$ if $i\neq 0, n+1$, while the restriction to $D_n^0$ yields an isomorphism of poset with $\mathsf{Div}(\Delta_{n-1})$, and the restriction to $D_n^{n+1}$ an isomorphism of posets with $\mathsf{Div}(\Delta_0)=\{\bullet\}$. In particular, by induction, all these posets are lattices. As the poset $(\Div(\Delta_n), \leq)$ is finite and admits a maximal element, it suffices to show that $x \wedge y$ as defined by the formula above is indeed the join of $x$ and $y$. 

It is clear that $x \wedge y \leq x$. Let us show that $x\wedge y \leq y$. If $i=j$ this is clear, hence assume that $i<j$. By~\cref{rho_n_j_entre} we see that $$\bigvee_i\{z\in D_n^i \ | \ z\leq y\}=\bigvee_i\{z\in D_n^i \ | \ z\leq \rho_n^j\}.$$
It suffices to check that $\bigvee_i\{z\in D_n^i \ | \ z\leq \rho_n^j\}\leq \rho_n^j$. Note that $$\bigvee_i\{z\in D_n^i \ | \ z\leq \rho_n^j\}=\bigvee_i\{z\in D_n^i \ | \ z\leq \rho_n^j\text{~and~}(z\leqdot x \leq \rho_n^j\Rightarrow x\notin D_n^i)\}.$$
Now by~\cref{rho_n_j_entre}, if $z\in D_n^i$ and $x$ is any element such that $z\leqdot x \leq \rho_n^j$ and $x\notin D_n^i$, then $x=\rho_n^{k}$ for some $k$ (necessarily smaller than or equal to $j$). It implies that $$\bigvee_i\{z\in D_n^i \ | \ z\leq \rho_n^j\}=\bigvee_i\{z\in D_n^i \ | \ z\leq \rho_n^j\text{~and~}z\leqdot \rho_n^k\text{~for some~}k\leq j\}.$$
By~\cref{sorties}, we have that $$\bigvee_i\{z\in D_n^i \ | \ z\leq \rho_n^j\text{~and~}z\leqdot \rho_n^k\text{~for some~}k\leq j\}$$ has to be an element of the set $\{z\in D_n^i \ | \ z\leq \rho_n^j\text{~and~}z\leqdot \rho_n^k\text{~for some~}k\leq j\}$, hence that it is in particular a left-divisor of $\rho_n^j$ (and hence of $y$).

Now assume that $u\leq x,y$. We can assume that $u\in D_n^i$, otherwise by~\cref{rho_n_j_entre} we have $u< \rho_n^i\leq x \wedge y$. As $u\leq y$, we have that $u\leq \bigvee_i\{z\in D_n^i \ | \ z\leq y\}$. And hence, that $u\leq x \wedge_i \left(\bigvee_i\{z\in D_n^i \ | \ z\leq y\}\right)=x\wedge y$.     
\end{proof}

\section{Schröder trees and words for the Garside element}\label{sec_Schröder}

\subsection{labelling of Schröder trees}

A \defn{rooted plane tree} is a tree embedded in the plane with one distinguished vertex called the \defn{root}. The vertices of degree $1$ are called the \defn{leaves} of the tree and the other vertices are called \defn{inner vertices}. One can consider rooted trees as directed graphs by orienting the edges from the root toward the leaves. If there is an oriented edge from a vertex $v$ to a vertex $w$, we say that $v$ is the \defn{parent} of $w$ and $w$ is a \defn{child} of $v$. As can be seen in \Cref{fig:Schroeder_trees}, we draw the trees with their root on the top and the leaves on the bottom. The planar embedding induces a total ordering (from left to right) on the children of each vertex, hence we can speak about the leftmost child of a vertex. 

Alternatively one has a useful recursive definition of a rooted plane tree: it is either the empty tree with no inner vertex and a single leaf or a tuple $T = (r,\mathrm{Tr})$ where $r$ is the root vertex and $\mathrm{Tr}$ is an ordered list of rooted plane trees. If $T$ is a tree with the first definition, the vertex $r$ is its root and the list $\mathrm{Tr}$ is the list of subtrees, ordered from left to right, obtained by removing the root $r$ and all the edges adjacent to $r$ in $T$. 

\begin{definition}
{\noindent}
\begin{enumerate}
\item A \defn{Schröder tree} is a rooted plane tree in which each inner vertex has at least two children.
\item A \defn{binary tree} is a rooted plane tree in which each inner vertex has exactly two children.
\item The \defn{size} of a tree is its number of leaves.
\item The \defn{height} of a tree is the number of vertices in a maximal chain of descendants. 
\item The Schröder tree on $n$ leaves in which every child of the root is a leaf is called the~\defn{Schröder bush}. We denote it by $\delta_n$. 
\item The Schröder tree given by the binary tree in which every right child (resp. every left child) is a leaf is called a~\defn{left comb} (resp. a~\defn{right comb}). 
\end{enumerate}
\end{definition}
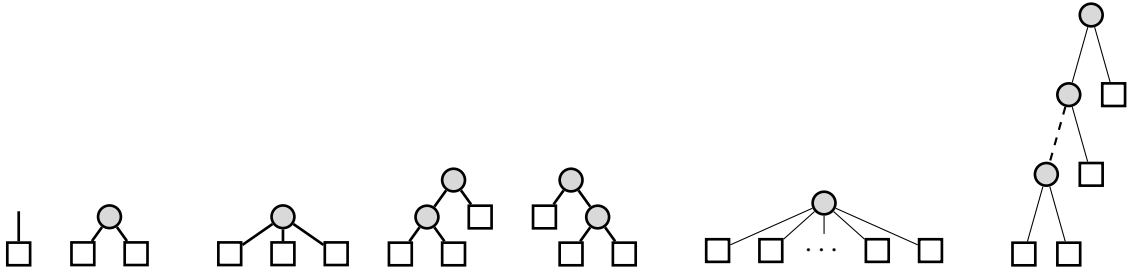
\begin{figure}[h]
\centering
\begin{tikzpicture}[scale = 0.7,
   level distance=10mm,
   level 1/.style={sibling distance=20mm},
   level 2/.style={sibling distance=15mm},
   level 3/.style={sibling distance=6mm},
   inner/.style={circle,draw=black,fill=black!15,inner sep=0pt,minimum size=4mm,line width=1pt},
   root/.style={},
   leaf/.style={rectangle,draw=black,inner sep=0pt,minimum size=3mm,line width=1pt},                      
   edge from parent/.style={draw,line width=1pt}]
  \node[root]{}
  	child {node [leaf] {}};
\end{tikzpicture}
\quad
\begin{tikzpicture}[scale = 0.7,
   level distance=7mm,
   level 1/.style={sibling distance=10mm},
   level 2/.style={sibling distance=15mm},
   level 3/.style={sibling distance=6mm},
   inner/.style={circle,draw=black,fill=black!15,inner sep=0pt,minimum size=3mm,line width=1pt},
   leaf/.style={rectangle,draw=black,inner sep=0pt,minimum size=3mm,line width=1pt},                      
   edge from parent/.style={draw,line width=1pt}]
  \node[inner]{}
  	child {node [leaf] {} }
	child {node [leaf] {}};
\end{tikzpicture}
\qquad
\begin{tikzpicture}[scale = 0.7,
   level distance=7mm,
   level 1/.style={sibling distance=10mm},
   level 2/.style={sibling distance=15mm},
   level 3/.style={sibling distance=6mm},
   inner/.style={circle,draw=black,fill=black!15,inner sep=0pt,minimum size=3mm,line width=1pt},
   leaf/.style={rectangle,draw=black,inner sep=0pt,minimum size=3mm,line width=1pt},                      
   edge from parent/.style={draw,line width=1pt}]
  \node[inner]{}
  	child {node [leaf] {}}
	child {node [leaf] {}}
	child {node [leaf] {}};
\end{tikzpicture}
\quad
\begin{tikzpicture}[scale = 0.7,
   level distance=7mm,
   level 1/.style={sibling distance=10mm},
   level 2/.style={sibling distance=10mm},
   level 3/.style={sibling distance=6mm},
   inner/.style={circle,draw=black,fill=black!15,inner sep=0pt,minimum size=3mm,line width=1pt},
   leaf/.style={rectangle,draw=black,inner sep=0pt,minimum size=3mm,line width=1pt},                      
   edge from parent/.style={draw,line width=1pt}]
  \node[inner]{}
  	child {node [inner] {}
		child {node [leaf] {}}
		child {node [leaf] {}}
		}
	child {node [leaf] {}};
\end{tikzpicture}
\quad 
\begin{tikzpicture}[scale = 0.7,
   level distance=7mm,
   level 1/.style={sibling distance=10mm},
   level 2/.style={sibling distance=10mm},
   level 3/.style={sibling distance=6mm},
   inner/.style={circle,draw=black,fill=black!15,inner sep=0pt,minimum size=3mm,line width=1pt},
   leaf/.style={rectangle,draw=black,inner sep=0pt,minimum size=3mm,line width=1pt},                      
   edge from parent/.style={draw,line width=1pt}]
  \node[inner]{}
  	child {node [leaf] {}}
  	child {node [inner] {}
		child {node [leaf] {}}
		child {node [leaf] {}}
		};
\end{tikzpicture}
\qquad
\begin{tikzpicture}[scale = 0.7,
   level distance=9mm,
   level 1/.style={sibling distance=10mm},
   level 2/.style={sibling distance=10mm},
   level 3/.style={sibling distance=5mm},
   inner/.style={circle,draw=black,fill=black!15,inner sep=0pt,minimum size=3mm,line width=1pt},
   leaf/.style={rectangle,draw=black,inner sep=0pt,minimum size=3mm,line width=1pt},                      
   cdots/.style={}                
   edge from parent/.style={draw,line width=1pt}] 
  \node[inner]{}
		child {node [leaf] {}}
		child {node [leaf] {}}
		child {node [cdots] {$\cdots$}}
		child {node [leaf] {}}
		child {node [leaf] {}};
\end{tikzpicture}
\qquad
\begin{forest}
  [,circle,draw=black,fill=black!15,inner sep=0pt,minimum size=3mm,line width=1pt, edge={line,very thick}
   [,circle,draw=black,fill=black!15,inner sep=0pt,minimum size=3mm,line width=1pt
    [,circle,draw=black,fill=black!15,inner sep=0pt,minimum size=3mm,line width=1pt,edge={dashed,thick}
    [,rectangle,draw=black,inner sep=0pt,minimum size=3mm,line width=1pt]
    [,rectangle,draw=black,inner sep=0pt,minimum size=3mm,line width=1pt]
    ]
     [,rectangle,draw=black,inner sep=0pt,minimum size=3mm,line width=1pt]
    ]
    [,rectangle,draw=black,inner sep=0pt,minimum size=3mm,line width=1pt]
     ]
\end{forest}
\caption{From left to right: the unique Schröder tree with $1$ leaf, the unique Schröder tree with two leaves, the three Schröder trees with $3$ leaves. Then the Schröder bush and on its right a left comb.}
\label{fig:Schroeder_trees}
\end{figure} 
The Schröder trees are counted by the so-called \defn{little Schröder numbers}. The sequence starts with $1, 1, 3, 11, 45, 197, 903, 4279,20793,...$ and is referred as \href{https://oeis.org/A001003}{A001003} in~\cite{oeis}.

We will label (and read the labels of) the vertices and the leaves of our trees using the so-called \emph{post-order traversal}. This is a recursive algorithm that visits each vertex and leaf of the tree exactly once. Concretely, if $T = \big(r,(T_1,\dots, T_k)\big)$ is a rooted planar tree, then we recursively apply the algorithm to $T_1$, $T_2$ until $T_k$ and finally we visit the root $r$.  When the algorithm meets an empty tree it visits its leaf and then, the recursion stops and it goes up one level in the recursive process.  The first vertex visited by the algorithm is the leftmost leaf of $T$, then the algorithm moves to its parent $v$ (but does not visit $v$) and visits the second subtree of $v$ starting with the leftmost leaf and so on. We refer to \Cref{fig:post_order} for an illustration where the first vertex visited by the algorithm is labeled by $1$, the second by $2$ and so on. The last vertex visited by the algorithm is always the root of $T$. 
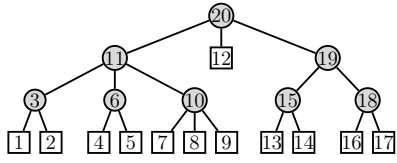
\begin{figure}[h]
\centering
\scalebox{.7}{\begin{tikzpicture}
  [ level distance=8mm,
   level 1/.style={sibling distance=20mm},
   level 2/.style={sibling distance=15mm},
   level 3/.style={sibling distance=6mm},
   inner/.style={circle,draw=black,fill=black!15,inner sep=0pt,minimum size=4mm,line width=1pt},
   leaf/.style={rectangle,draw=black,inner sep=0pt,minimum size=4mm,line width=1pt},                      
   edge from parent/.style={draw,line width=1pt}]
  \node [inner] {20}
     child {node [inner] {11}
       child {node [inner] {3}
         child {node [leaf] {1}}
         child {node [leaf] {2}}
       }
       child {node [inner] {6}
         child {node [leaf] {4}}
         child {node [leaf] {5}}
       }
       child {node [inner] {10}
         child {node [leaf] {7}}
         child {node [leaf] {8}}
         child {node [leaf] {9}}
       }
     }
     child {node [leaf] {12}}
     child {node [inner] {19}
       child {node [inner] {15}
         child {node [leaf] {13}}
         child {node [leaf] {14}}
       }
       child {node [inner] {18}
         child {node [leaf] {16}}
         child {node [leaf] {17}}
       }
     };
\end{tikzpicture}}
\caption{Post-order traversal of a Schröder tree of size 12.}
\label{fig:post_order}
\end{figure}
Let $m,n$ be two integers such that $m\geq n-1$. We then label a Schröder tree $T$ with $n\geq 2$ leaves by labelling its vertices one after the other with respect to the total order defined by the post-order traversal, using the following rules: 
\begin{enumerate}
\item Let $v$ be the leftmost child of a vertex $w$. Then $w$ is the root of a Schröder tree $\big(w,(T_1,\cdots,  T_k)\big)$ and $v$ is the root of $T_1$.  The label $\lambda(v)$ of $v$ is equal to the number of leaves of the forest consisting of all the trees $T_2,\cdots ,T_k$.  
\item If $v$ is not the leftmost child of a vertex of $T$, we consider $LD(v)$ the set of its leftmost descendants consisting of the leftmost child of $v$ and its leftmost child and so one. Then the  label of $v$ is $m - \sum_{w\in LD(v)} \lambda(w).$ Note that using the post-order traversal, the label of the leftmost descendants of a vertex $v$ are already determined when we visit $v$. 
\end{enumerate}

The result is a labelled Schröder tree that we denote by $L_{m}(T)$. This procedure is illustrated in \Cref{fig:label_tree}. 

\begin{definition}
Let $L_{m}(T)$ be a labelled Schröder tree with $n$ leaves labelled by $m\geq n-1$. The sum of the labels of the vertices of $T$ is called its \defn{weight} (with respect to $m$). 
\end{definition}

\begin{lemma}
Let $T$ be a Schröder tree with $n$ leaves and $m\geq n-1$. Then the integers labelling $L_{m}(T)$ are strictly nonnegative with the exception of the root which may be labelled by $0$. 
\end{lemma}
\begin{proof}
If a vertex is a leftmost child, then its label is a number of leaves,  hence it is positive. If $v$ is not a leftmost child, then it is labelled by $m - \sum_{w\in LD(v)} \lambda(w)$. Each $\lambda(w)$ is equal to a certain number of leaves of $T$ and the set of leaves associated to distinct vertices of $LD(v)$ do not intersect. Moreover, exactly one element of $LD(v)$ is a leaf and this leaf is not counted in $\sum_{w\in LD(v)} \lambda(w)$. We therefore have \begin{equation}\label{eq_1}\left(\sum_{w\in LD(v)} \lambda(w)\right)+1\leq n,\end{equation} 
hence $m-\sum_{w\in LD(v)}\lambda(w) \geq 0$. Moreover if $\sum_{w\in LD(v)} \lambda(w) = m$, then by~\eqref{eq_1} we must have $m=n-1$. It follows that $v$ has $n$ descendants since the leftmost leaf which is a descendant of $v$ is not counted, hence $v$ is the root of $T$. 
\end{proof}
This labelling is almost determined by the recursive structure of the tree, as shown by the following result. 
\begin{lemma}\label{lem:inductive_label}
Let $T = \big(r,(T_1,\dots, T_k) \big)$ be a Schröder tree and $v$ be a vertex of $T_i$ for $i\in \{1,\dots,k\}$. Then,
\begin{enumerate}
\item If $v$ is not the root of $T_1$, then its label in $L_{m}(T)$ is equal to its label in $L_{m}(T_i)$. 
\item If $v$ is the root of $T_1$, then its label in $L_{m}(T_1)$ is equal to the sum of the labels of $v$ and of the root of $T$ in $L_{m}(T)$. 
\end{enumerate}
\end{lemma}
\begin{proof}
Let $v$ be a vertex of $T_i$. If $v$ is a leftmost child in $T$ which is not the root of $T_1$, then its label is a number of leaves of a certain forest which is contained in $T_i$. Hence this number is the same in the big tree $T$ or in the extracted tree $T_i$. If $v$ is not a leftmost child, then its label is determined by the labels of its leftmost descendants, hence it is the same in the tree $T$ as in the extracted tree $T_i$ since we have just shown that the labels of leftmost descendants which are not the root of $T_1$ agree. The root of $T_1$ has a different behaviour since in $T$ it is a leftmost child and this is not the case in $T_1$. Hence if $v$ is the root of $T_1$, denoting by $\lambda_1$ the label of $v$ in $T_1$, we have $\lambda_1(v)=m- \sum_{w\in LD(v)} \lambda_1(w)$. The labels of the descendants of $v$ are the same in $T$ and in $T_1$, that is, we have $\lambda_1(w)=\lambda(w)$ for all $w\in LD(v)$. In $T$, the label of the root $r$ is given by $$\lambda(r)=m-\lambda(v)-\sum_{w\in LD(v)} \lambda(w)=m-\lambda(v)-\sum_{w\in LD(v)} \lambda_1(w).$$ Hence we have $\lambda(r)+\lambda(v) = \lambda_1(v)$. 
\end{proof}

\subsection{Words for the Garside element in terms of Schröder trees}

 Reading the labelled tree $L_{m}(T)$ using the post-order traversal and associating the generator $\rho_i$ to the letter $i$ with the convention that $\rho_0 = e$, gives a map $\Phi_{m}$ from the set of Schröder trees labelled by $m$ to the set $S^\star$ of words for the elements of the monoid $M_m$. We refer to \Cref{fig:label_tree} for an illustration.  
 
 \begin{figure}[h]
\centering
\scalebox{.7}{\begin{tikzpicture}
  [ level distance=8mm,
   level 1/.style={sibling distance=20mm},
   level 2/.style={sibling distance=15mm},
   level 3/.style={sibling distance=6mm},
   inner/.style={circle,draw=black,fill=black!15,inner sep=0pt,minimum size=4mm,line width=1pt},
   leaf/.style={rectangle,draw=black,inner sep=0pt,minimum size=4mm,line width=1pt},                      
   edge from parent/.style={draw,line width=1pt}]
  \node [inner] {0}
     child {node [inner] {5}
       child {node [inner] {5}
         child {node [leaf] {1}}
         child {node [leaf] {11}}
       }
       child {node [inner] {10}
         child {node [leaf] {1}}
         child {node [leaf] {11}}
       }
       child {node [inner] {9}
         child {node [leaf] {2}}
         child {node [leaf] {11}}
         child {node [leaf] {11}}
       }
     }
     child {node [leaf] {11}}
     child {node [inner] {8}
       child {node [inner] {2}
         child {node [leaf] {1}}
         child {node [leaf] {11}}
       }
       child {node [inner] {10}
         child {node [leaf] {1}}
         child {node [leaf] {11}}
       }
     };
\end{tikzpicture}}
\caption{Example of the labelling of a Schröder tree of size 12 with $m=11$. The corresponding element in the monoid $M_{11}$ is $\rho_1 \rho_{11} \rho_5 \rho_1 \rho_{11}\rho_{10}\rho_2 \rho_{11} \rho_{11} \rho_9 \rho_5 \rho_{11} \rho_1 \rho_{11} \rho_2 \rho_{1} \rho_{11} \rho_{10} \rho_8.$}
\label{fig:label_tree}
\end{figure}
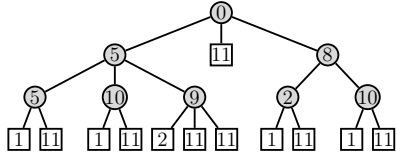

\begin{definition}
Let $T$ be a non-empty Schröder tree. If $T$ has a subtree $T_1$ satisfying the three following properties: 
\begin{enumerate}
\item The root $r_1$ of $T_1$ is not the root of $T$, hence it has a parent $r_0$ which has at least two children, 
\item The root $r_1$ has exactly two children, 
\item The right subtree of $T_1$ is the empty tree with only one leave.
\end{enumerate}
Then, we can construct another tree $\widetilde{T}$ by contracting the edge $r_0 - r_1$, in other words by removing the root $r_1$ of $T_1$ and attaching the two subtrees of $T_1$ to $r_0$. See \Cref{fig:local_move} for an illustration. We call such a transformation, or the inverse transformation, a~\defn{local move}.

Note that, since $r_0$ has at least two children in the configuration described above (see also the left picture in~\Cref{fig:local_move}), we get that $r_0$ has at least three children in the configuration obtained after applying the local move. In particular, to apply a local move in the other direction, we need to have a Schröder tree $\widetilde{T}$ with a subtree $T_1$ satisfying : \begin{enumerate}
	\item The parent $r_0$ of $T_1$ (which is allowed to be the root of $T$) has at least three children, 
	\item The tree $T_1$ is not the last child of $r_0$, and is directly followed by an empty tree with only one leaf. 
\end{enumerate}

\end{definition}

\begin{figure}[h!]
\centering
\scalebox{.7}{\begin{tikzpicture}
  [ level distance=8mm,
   level 1/.style={sibling distance=30mm},
   level 2/.style={sibling distance=25mm},
   level 3/.style={sibling distance=20mm},
   inner/.style={circle,draw=black,fill=black!15,inner sep=0pt,minimum size=4mm,line width=1pt},
   leaf/.style={rectangle,draw=black,inner sep=0pt,minimum size=4mm,line width=1pt},  
   dots/.style={star,draw=black,inner sep=0pt,minimum size=8mm, text width=7mm,align=center, line width=1pt,fill=black!15},  
   root/.style={}               
   edge from parent/.style={draw,line width=1pt}]

\node [root] {}
     child {node[inner] {$r_0$}
     		child{node[dots] {$S_k$}}
		child{node[inner]{$r_1$}
			child{ node[inner] {$r_2$} 
				child{node[dots] {$A_1$} edge from parent[draw=none] }}
			child{node[leaf] {}}
			}
		child{node[dots] {$S_{k+2}$}}
       };
\end{tikzpicture}}
\quad $\longleftrightarrow$
\quad
\scalebox{.7}{\begin{tikzpicture}
  [ level distance=15mm,
   level 1/.style={sibling distance=20mm,level distance=8mm},
   level 2/.style={sibling distance=25mm,level distance=8mm},
   inner/.style={circle,draw=black,fill=black!15,inner sep=0pt,minimum size=4mm,line width=1pt},
   leaf/.style={rectangle,draw=black,inner sep=0pt,minimum size=4mm,line width=1pt},  
   dots/.style={star,draw=black,inner sep=0pt,minimum size=8mm, text width=7mm,align=center, line width=1pt,fill=black!15},  
   root/.style={}               
   edge from parent/.style={draw,line width=1pt}]

\node [root] {}
     child {node[inner] {$r_0$}
     		child{node[dots] {$S_k$}}
		child{ node[inner] {$r_2$} 
				child{node[dots] {$A_1$} edge from parent[draw=none] }}
		child{node[leaf] {}}
		child{node[dots] {$S_{k+2}$}}
       };
\end{tikzpicture}}
\caption{Local move.}\label{fig:local_move}
\end{figure}
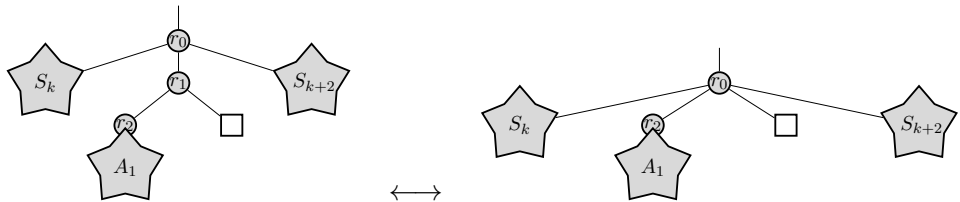

 Formally, if the subtree of $T$ with root $r_0$ is $S = \big(r_0, (S_1,\cdots, S_k,T_1,S_{k+2},\cdots, S_r)\big)$ and the subtree $T_1$ is $\big(r_1,(A_1, A_2)\big)$, then we obtain the tree $\widetilde{T}$ by replacing $S$ by 
 \[ \big(r_0, (S_1,\cdots, S_k,A_1, A_2,S_{k+2},\cdots, S_r)\big). \]

 \begin{lemma}\label{lem:local_move}
Let $T$ and $S$ be two Schröder trees with $n$ leaves. Then one can pass from the tree $T$ to the tree $S$ by applying a sequence of local moves. 
\end{lemma}
\begin{proof}
It is enough to show that $T$ can be transformed into the Schröder bush $\delta_n$--recall that this is the Schröder tree in which every child of the root is a leaf--by a sequence of local moves. The Schröder tree $S$ can then be transformed as well into $\delta_n$, and hence $T$ can be transformed into $S$. We argue by induction on the number of leaves. For $n=1$ and $n=2$ there is nothing to prove. If $T$ has a subtree $S$ which is not of the form $\delta_k$, then we can transform $S$ into $\delta_k$ for some $k$ by applying the induction hypothesis to $S$. Hence we can assume that $T = \big(r,(\delta_{n_1},\cdots, \delta_{n_k})\big)$ with $\sum n_i =n$. If $T$ is not equal to $\delta_n$, then it has at least a non-empty subtree $S$. If $S$ has only two leaves, then we can apply a local move to remove its root and to attach the two leaves to the root of $T$. If it has more than $3$ leaves, by induction there is a sequence of local moves from $S$ to a left comb. Then, by repeatedly applying a local move at the root of the left comb, we remove all the inner vertices of the left comb and attach all its leaves to the root of $T$. Applying this to all subtrees $S$ of $T$ wich are not empty, we end up getting $\delta_n$.  
\end{proof}

\begin{lemma}\label{lem:b_def}
Let $T$ be a Schröder tree with $n$ leaves and $m\geq n-1$. Then $\Phi_{m}(T)$ is a word for $\rho_{n-1} (\rho_{m})^{n-1} \rho_{m-n+1}$ in $M_{m}$. In particular if $m = n-1$, then it is a word for the Garside element of $M_{n-1}$. 
\end{lemma}
\begin{proof}
If $T = \delta_n$ is the Schröder tree with only one root and $n$ leaves, then $\Phi_m(T) = \rho_{n-1} (\rho_m)^{n-1} \rho_{m-n+1}$. If $T$ is another Schröder tree, then by \Cref{lem:local_move} there is a sequence of local moves from $T$ to $\delta_n$. To finish the proof it is enough to show that applying a local move to a Schröder tree $T$ amounts to applying a relation of the monoid $M_m$ to $\Phi_m(T)$. This is easily obtained by staring at~\Cref{fig:local_move}.

Indeed, if $T$ is the tree at the left of \Cref{fig:local_move}, then the label of $r_2$ is $1$, the label of the leaf on its right is $m$ and the label of $r_1$ is a certain integer $\ell$. Since $r_1$ is not the root of $T$, we have $1\leq \ell$. Moreover, since $r_1$ is not a leaf of $T$, we have $\ell<m$. Hence in $\Phi_{m}(T)$ we have the factor $\rho_{1} \rho_m \rho_{\ell}$ with $1\leq \ell \leq m-1$. 

If $\widetilde{T}$ denotes the right tree of \Cref{fig:local_move}, then the label of $r_2$ is $\ell+1$. Indeed $r_2$ is a leftmost child in $\widetilde{T}$ if and only if $r_1$ is a leftmost child in $T$. In this case its label is the number of leaves of the forest in its right and in $\widetilde{T}$ there is precisely one more leaf in this forest than in $T$.  In the other case, the label of $r_2$ in $\widetilde{T}$ is $m - \sum_{w\in LD(r_2)} \lambda(w)$. The label of $r_1$ is $\ell = m - 1 - \sum_{w \in LD(r_2)} \lambda(w)$. So the label of $r_2$ is $\ell+1$. The leaf on the right of $r_2$ in $\widetilde{T}$ is labelled by $m$, hence $\Phi_{m}(\widetilde{T})$ is obtained by replacing $\rho_{1} \rho_m \rho_{\ell}$ in $\Phi_m(T)$ by $\rho_{\ell+1}\rho_m$, and vice-versa.  \end{proof}

\begin{prop}\label{prop:surj}
For $m=n-1$, the map $\Phi_{m}$ from the set of Schröder trees with $n$ leaves to the set of words for $\rho_{n-1}^n$ in $M_{n-1}$ is surjective.
\end{prop}

\begin{proof}
We have to show that to each word $y$ for $\rho_{n-1}^n\in M_{n-1}$, we can attach a Schröder tree $T$ with $n$ leaves, in such a way that $\Phi_m(T)=y$. The word $y$ and the word $\rho_{n-1}^n$ can be transformed into each other by applying a sequence of defining relations of $M_m$. 
We already know that the word $\rho_{n-1}^n$ is in the image of $\Phi_m$ since it is the image of the Schröder bush. To conclude the proof, we therefore need to show the following claim: given a Schröder tree $S$, if the corresponding labelling has a substring of the form $1 m \ell$ (resp. $(\ell+1)m$) with $1\leq \ell \leq m-1$, then we are necessarily in the configuration of the left picture in Figure~\ref{fig:local_move} (resp. the right picture), and hence we can apply a local move. Indeed, as one can pass from the word $\rho_{n-1}^n$ to the word $y$ by a sequence of defining relations let $y_0=\rho_{n-1}^n, y_1, \dots, y_k=y$ be expressions of $\rho_{n-1}^n$ such that $y_i$ is obtained from $y_{i-1}$ by applying a single relation in $M_m$. Applying the relation on $y_0=\Phi_m(T)$ to get $y_1$ corresponds to applying a local move on $T$ to get a Schröder tree $T_1$ and as seen in the proof of~\cref{lem:b_def}, we get $\Phi_m(T_1)=y_1$.

To show the claim, assume that $S$ is a Schröder tree with labelling having a substring of the form $1 m \ell$ with $1\leq \ell \leq m-1$. Note that $m$ can only be the label of a leaf. Let $v$ be the parent of that leaf. It is a root of a family of trees, say $(v, T_1, \dots, T_k)$ and our leaf with label $m$ corresponds to one of the trees $T_i$ (which has to be empty). It is clear that such a tree cannot be $T_1$: indeed, as $T_1$ is the leftmost child of $v$, in that case $m=n-1$ would be the number of leafs in the forest $T_2, \dots, T_k$, which is at most $n-1$. As $m= n-1$, the only possibility would be that $v$ is the root of $S$, hence $m$ would be the first label and therefore could not be preceded by a label $1$. Hence $m$ labels one of the trees $T_2, \dots, T_k$, say $T_i$. It follows that the label $1$ preceding $m$ is the label of the root of $T_{i-1}$. If $i=2$ then $k=2$ as the label $1$ is then the label of the leftmost child of $v$, meaning that there is only one leaf in the forest $T_2, \dots, T_k$. In that case, it only remains to show that $v$ cannot be the root of $S$ to match the configuration in the left picture of Figure~\ref{fig:local_move}. But this is clear for if $v$ was the root of $S$, the last label would be $m$ corresponding to $T_2$, hence no $\ell$ could appear. Hence $v$ is not the root of $S$, and its label is $\ell$. Now if $i\neq 2$, then $i-1\neq 1$. The root $v'$ of $T_{i-1}$ is labelled by $1$ and as $v'$ is not the leftmost child of $v$, we have $1=\lambda(v')=m-\sum_{w\in LD(v')} \lambda(w)$, yielding $\sum_{w\in LD(v')} \lambda(w)=m-1$. This means that there are $m$ leaves in $T_{i-1}$, and as there is one leaf in $T_i$ and $m= n-1$, the only possibility is that $i-1=1$ and $k=2$, contradicting $i\neq 2$. 

Now, assume that $S$ is a Schröder tree with labelling having a substring of the form $(\ell+1) m$ with $1\leq \ell \leq m-1$. Again, $m$ can only label a leaf. Let $v$ be the parent of that leaf as above, which is a root of a family $T_1, \dots, T_k$ of trees with $T_i$ corresponding to our leaf for some $i$. We need to show that $i\neq 1$ and $k\geq 3$. In the previous case we have seen that if $i=1$, then $m=n-1$ is the number of leaves in $T_2, \dots, T_k$, forcing $v$ to be the root of $S$ and $m$ to be the first label in $S$. Hence $i\geq 2$. If $k=2$ (hence $i=2$), then the root of $T_1$ is labelled by $1=\ell+1$, contradicting $1 \leq \ell$. Hence $k\geq 2$.       
\end{proof}

\begin{lemma}\label{lem:weight_tree}
\begin{enumerate}
\item Let $T$ be a Schröder tree with $n$ leaves labelled by $m\geq n-1$. Then, the weight of $T$ is $nm$. 
\item Let $w$ be a vertex of $T$ which is not a leaf and $v$ its leftmost child, that is $w$ is the root of a Schröder tree $\big(w,(T_1,\cdots,  T_k)\big)$ and $v$ is the root of $T_1$. Then the weight of the forest $F = (T_2,\cdots,T_k)$ attached to $w$ is $\lambda(v)m$, and the labelling of of a vertex in a tree $T_i$ for $i\geq 2$ is the same as its labelling inside $T$. 
\end{enumerate}
\end{lemma}
\begin{proof}
The first result is proved by induction on the number of leaves. If the tree has one leaf the result holds by definition of our labelling. Let $T = (r,T_1,\cdots ,T_k)$ be a Schröder tree, where $T_i$ has $n_i$ leaves. By induction, the tree $T_i$ has weight $mn_i$ for $i\geq 1$. Using~\Cref{lem:inductive_label}, the sum of the labels of the vertices of the tree $T_i$ (in $T$) is equal to $mn_i$ for $i\geq 2$ and the sum of the labels of the vertices of $T_1$ and of the root of $T$ is equal to $mn_1$. Hence, the tree $T$ has weight $\sum_{i = 1}^{k} mn_i = mn$. For the second point, the number of leaves of the forest $F$ is equal to $\lambda(v)$. Hence by the first point, the forest $F$ has weight $\lambda(v)m$. 
\end{proof}
\begin{prop}\label{prop:label_injectivity}
Let $m\geq n-1$. Then the map $\Phi_m$ from the set of Schröder trees with $n$ leaves to the set of words for the element $\rho_{n-1} (\rho_{m})^{n-1} \rho_{m-n+1}\in M_m$ is injective. 
\end{prop}
\begin{proof}
Let $T = (r,T_1,\cdots, T_k)$ be a Schröder tree with $n$ leaves labelled by $m$. This proof is purely combinatorial and it only involves the word $W$ in $\mathbb{N}$ obtained by reading the labels of the tree in post-order. The first step of the proof is to remark that one can recover the decomposition `root and list of subtrees' of a Schröder tree just by looking at $W$. We will illustrate the algorithm in~\Cref{ex:decomp} below. Precisely we want to split the word $W$ into a certain number of factors $W = W_1 \cdots W_k$ such that each subword $W_i$ is equal to the word obtained by reading the labels of $L_{m}(T_i)$ in post-order.

The first letter $w_1$ of $W$ is the label of the leftmost leaf $r_1$ of $T$ and by induction we will find the letters $w_2, w_3,\cdots,w_i$ corresponding to the ancestors $r_2,r_3,\cdots, r_i$ of $r_1$. Since the labels of these vertices count a number of leaves of $T$, when $\sum_{j=1}^{i} w_j = n-1$, then all the leaves of $T$ have been counted so $r_i$ is the root of $T_1$ and we stop the induction.

If we have found the letter $w_k$ corresponding to $r_k\neq r$,  then $w_k$ is the number of leaves of the right forest attached to the parent $r_{k+1}$ of $r_k$. By \Cref{lem:weight_tree}, the weight of $F$ is $m\cdot w_k$, hence the word obtained by reading the vertices of $F$ is $w_{k+1}\cdots w_i$ where $i$ is the smallest integer such that $\sum_{j = k+1}^{i} w_j= mw_k$. All these letters correspond to the vertices of $F$, hence the next letter is the label of the vertex read after $F$ in the post-order traversal, which is the vertex $r_{k+1}$. 
 
Since the word $W$ only contains strictly non-negative integers (except possibly the label of the root of $T$), at each step of the induction the value $w_1 + \cdots +w_i$ strictly increases and the induction stops. If $w_{n_1}$ is the letter corresponding to the leftmost child of the root of $T$, then the word $w_1\cdots w_{n_1}$ is the word obtained by reading all the vertices of the subtree $T_1$. By \Cref{lem:inductive_label}, this is almost the word obtained by reading $L_{m}(T_1)$ we just need to `correct' the label of the root of $T_1$ by adding the label of the root of $T$ which is the last letter $w_l$ of $W$. To conclude the word consisting of the labels of $T_1$ is $W_{T_1} = w_1 \cdots w_{n_{1}-1} (w_{n_1} + w_l)$. 
 
 Let $\widetilde{W}$ be the word obtained by removing the letters $w_1,\cdots, w_{n_1}$ and $w_l$. 
 We use the same procedure to extract the subwords corresponding to the other subtrees of $T$. Due to the asymmetry of \Cref{lem:inductive_label}, there is a slight difference. We have found all the labels $w_1,w_2,\cdots, w_t$ of the vertices $r_1,r_2,\cdots, r_t$ of the left branch of $T_i$ when $\sum_{j = 1}^{t}w_j = m$ and there is no need to `correct' the word as above.

We are now ready to prove that $\Phi_{m}$ is injective. If the words of two trees $T = (r, (T_1,\cdots, T_k)) $ and $S = (s, (S_1,\cdots, S_l))$ obtained by reading the labels of their vertices in post-order are equal, then by the discussion above we have $ k = l $ and for $i \in \{1,\cdots, k\}$, the words obtained by reading the vertices of the subtrees $L_m(T_i)$ and $L_{m}(S_i)$ are equal. By induction on the number of leaves, we have $S_i = T_i$ for $i = 1,\cdots,k$ and we get that $T = S$. 
\end{proof}
\begin{exple}\label{ex:decomp}
We illustrate the decomposition involved in the proof of \Cref{prop:label_injectivity} with the example of \Cref{fig:label_tree}. We consider the leftmost subtree $T_1$ of $T$ with $n=7$ leaves and which is labelled by $m=11$. We have $\Phi_{11}(T_1) = \rho_1 \rho_{11} \rho_{5} \rho_1 \rho_{11}\rho_{10} \rho_{2} \rho_{11} \rho_{11} \rho_{9} \rho_{5}$. The first letter $1$ tels us that the forest on the right of the leftmost leaf $r_1$ has $1$ vertex. Its weight is $m = 11$. Hence $\rho_{11}$ labels the only vertex of the forest and the next letter $5$ corresponds to the parent $r_2$ of $r_1$. Since $1+5 = 6$ we know that it is the leftmost child of the root. Hence the word $\rho_1 \rho_{11} \rho_{5}$ is obtained by reading the vertices of the leftmost subtree $S$ of $T_1$. We apply the `correction' and we get $\rho_{1} \rho_{11} \rho_{10} = \Phi_{11}(S)$. The rest of the word $\rho_1 \rho_{11}\rho_{10} \rho_{2} \rho_{11} \rho_{11} \rho_{9}$ corresponds to the other subtrees of $T_1$ and it splits as $\rho_1 \rho_{11}\rho_{10}$ and $\rho_{2} \rho_{11} \rho_{11} \rho_{9}$.
\end{exple}

Combining~\Cref{prop:surj} and~\Cref{prop:label_injectivity} we get our main result of the section:

\begin{theorem}\label{thm_main_reduced}
For $m=n-1$, the map $\Phi_{m}$ from the set of Schröder trees with $n$ leaves to the set of words for $\rho_{n-1}^n$ in $M_{n-1}$ is bijective.
\end{theorem}

\begin{cor}\label{cor_graph}
The following two graphs are isomorphic under $\Phi_{n-1}$:
\begin{enumerate}
\item The graph of words for $\rho_{n-1}^{n}$ in $M_{n-1}$, where vertices are given by expressions of $\rho_{n-1}^{n}$ and there is an edge between two expressions whenever they differ by application of a single relation,
\item The graph of Schröder trees with $n$ leaves, where vertices are given by Schröder trees and there is an edge between two trees whenever they differ by application of a local move. 
\end{enumerate} 
\end{cor}

\begin{proof}
The previous theorem gives the bijection between the sets of vertices. The proof of \Cref{lem:b_def} shows that whenever one can apply a local move, one can apply a relation on the corresponding words. The proof of \Cref{prop:surj} shows that whenever one can apply a relation on words, a local move can be applied on the corresponding trees.  
\end{proof}

We illustrate the situation for $M_3$ in Figure~\ref{words_garside_n_3} below. 

\begin{cor}\label{cor_words_g}
The number of words for the Garside element of $M_n$ is a little Schröder number \href{https://oeis.org/A001003}{A001003} \cite{oeis}.
\end{cor}

\begin{figure}[h!]
	\begin{center}
		\begin{tikzpicture}
			\matrix (m) [matrix of math nodes,row sep=3em,column sep=4em,minimum width=2em]
			{
				\rho_2 \rho_1 \rho_3\rho_2 \rho_1 \rho_3 & \rho_3 \rho_1 \rho_3 \rho_2 \rho_3 & \rho_3 \rho_1 \rho_3 \rho_1 \rho_3 \rho_1 & \rho_3 \rho_2 \rho_3\rho_3 \rho_1 \\
				\rho_2 \rho_3\rho_3 \rho_1 \rho_3 & \rho_3 \rho_3 \rho_3 \rho_3 & \rho_3 \rho_3 \rho_1\rho_3 \rho_2 & \rho_3 \rho_2 \rho_1 \rho_3 \rho_2 \rho_1 \\
				\rho_1\rho_3 \rho_1 \rho_3 \rho_1 \rho_3 & \rho_1 \rho_3 \rho_2 \rho_3 \rho_3 & \rho_1 \rho_3 \rho_2 \rho_1 \rho_3 \rho_2 &  \\};
			\path[-]
			(m-1-1) edge (m-2-1)
			(m-2-1) edge (m-3-1)
			(m-3-1) edge (m-3-2)
			(m-3-2) edge (m-2-2)
			(m-1-2) edge (m-2-2)
			(m-1-2) edge (m-1-3)
			(m-3-2) edge (m-3-3)
			(m-3-3) edge (m-2-3)
			(m-1-4) edge (m-1-3)
			(m-2-4) edge (m-1-4)
			(m-2-2) edge (m-2-3);
		\end{tikzpicture}
	\end{center} 
\vspace{1.2cm}
\begin{center}
		\begin{tikzpicture}
	\matrix (m) [matrix of math nodes,row sep=3em,column sep=4em,minimum width=2em]
	{
		  \Tree[. [.1 [.2 ]
		  [.2 [.1 ] [.3 ] ]]
		  [.3 ]] & \Tree[. [.3 ] [.2 [.1 ] [.3 ]] [.3 ]] & \Tree[. [.3 ] [.1 [.1 [.1 ] [.3 ]] [.3 ]] ] & \Tree[. [.3 ] [.1 [.2 ] [.3 ] [.3 ]]] \\
		\Tree[. [.1 [.2 ] [.3 ] [.3 ]] [.3 ]] & \Tree[. [.3 ] [.3 ] [.3 ] [.3 ] ] & \Tree[. [.3 ] [.3 ] [.2 [.1 ] [.3 ]] ] & \Tree[. [.3 ] [.1 [.2 ] [.2 [.1 ] [.3  ]]]] \\
		\Tree[. [.1 [.1 [.1 ] [.3 ]] [.3 ]] [.3 ]] & \Tree[. [.2 [.1 ] [.3 ]] [.3 ] [.3 ]] & \Tree[. [.2 [.1 ] [.3 ]] [.2 [.1 ] [.3 ]]] &  \\};
	\path[-]
	(m-1-1) edge (m-2-1)
	(m-2-1) edge (m-3-1)
	(m-3-1) edge (m-3-2)
	(m-3-2) edge (m-2-2)
	(m-1-2) edge (m-2-2)
	(m-1-2) edge (m-1-3)
	(m-3-2) edge (m-3-3)
	(m-3-3) edge (m-2-3)
	(m-1-4) edge (m-1-3)
	(m-2-4) edge (m-1-4)
	(m-2-2) edge (m-2-3);
\end{tikzpicture}	
\end{center}
	\caption{Illustration of Corollary~\ref{cor_graph} for $n=4$: the graph of reduced words for $\Delta_3$ and the isomorphic graph of Schröder trees on $4=3+1$ leaves.}
	\label{words_garside_n_3}
\end{figure}

\begin{lemma}\label{lem:wordsubtree}
Let $T = (r,S_1,\cdots, S_k)$ be a Schröder tree with $n$ leaves labelled by $m = n-1$. Then, the word obtained by reading all the labels of a subtree $S_j$ is a word for $\rho_n^{l_j}$ where $l_j$ is the number of leaves of $S_j$.
\end{lemma}
\begin{proof}
Let us assume that the tree $S_j$ has $s+1$ leaves. By Lemma \ref{lem:b_def}, the labels of the subtree $S_j$ is a word for $\rho_{s} \rho_{n-1}^{s} \rho_{n-1-s}$. If $s=0$, then we have a word for $\rho_{n-1}$. Otherwise, we can apply the relations \cite[Lemma 4.5]{Gobet} with $i=s$ and $j = n-1$. Alternatively, using the Schröder trees, it is easy to see that these relations comes from the following modifications of the trees. The word $\rho_{s} \rho_{n-1}^{s} \rho_{n-1-s}$ correspond to the case where the tree $S_j$ is the Schröder bush with $s+1$ leaves. Using our local moves, we can modify it to the left comb. The corresponding word is now $(\rho_1 \rho_s)^s\rho_{n-1-s}$. Now we can inductively apply the local move to contract the edge between the root of $T$ and the root left comb. The result is $s+1$ empty trees attached to the root of $T$ and the corresponding word is $\rho_{n-1}^{s+1}$. 
\end{proof}

\begin{prop}\label{dni_trees}
Let $n\geq 1$. We have the following isomorphisms of posets: 
\begin{enumerate}
	\item $D_n^{0}\cong\mathsf{Div}(\Delta_{n-1})$, $D_n^{n+1}\cong \mathsf{Div}(\Delta_0)=\{\bullet\}$,
	\item For all $1\leq i \leq n$, $D_n^i \cong  \mathsf{Div}(\Delta_{n-i})$, 
\end{enumerate}	where every set is ordered by the restriction of the left-divisibility order in the monoid $M_k$ for suitable $k$. 
\end{prop}

\begin{proof}
We begin by proving the second statement. An element $x$ of $D_n^i$ can be written in the form $\rho_n^i x'$, where $x'$ is uniquely determined by cancellativity, and such that $\rho_n$ is not a left-divisor of $x'$. In particular, there is $y$ a divisor of $\Delta_n$ such that $\rho_n^i x' y =\rho_n^{n+1}$, and $y\neq 1$. We associate a tree (or rather a family of trees) to $x$ as follows. Write $x'$ as a product $a_1 a_2 \cdots a_j$ of elements of $\mathcal{S}$. Complete the word $\rho_n^i a_1 a_2 \cdots a_j$ to a word $\rho_n^i a_1 a_2 \cdots a_jb_1 b_2 \cdots b_\ell$ for $\Delta_n$, i.e., choose a word $b_1 b_2\cdots b_\ell$ for $y$. There are several possibilités for the $b_i$'s, but the condition that $x\in D_n^i$ ensures that, writing the corresponding Schröder tree in the form $(r, T_1, T_2, \dots, T_i, S_1, S_2, \cdots S_d)$, where the $i$ first trees are empty trees with a single leaf, then $a_1 a_2 \cdots a_j$ has all its labels inside $S_1$. Indeed, the labelling $a_1a_2\cdots a_j$ begins at the beginning (in the post-order convention) of the tree $S_1$ since the trees $T_1, T_2, \dots, T_i$ yield the label $\rho_n^i$, and if another tree among $S_2, \dots, S_d$ was partly labelled by the $a_i$'s, then a power of $\rho_n$ would left-divide $x'$, since the word obtained from $S_1$ is a power of $\rho_n$ (lemma \ref{lem:wordsubtree}). It is then possible to reduce all the trees $S_1, S_2, \dots, S_d$ to a single tree $S$ still having the labelling $a_1, a_2, \dots, a_j$ at the beginning, by first reducing $S_2, \dots, S_d$ to a set of empty trees $\widetilde{T}_2, \dots, \widetilde{T}_{d'}$ with single leafs, and then merging $S_1$ and $\widetilde{T}_2$ using a local move, then merging the resulting tree with $\widetilde{T}_3$, and so on (see Figure~\ref{fig:merge} for an illustration). In this way we associate to $x$ a Schröder tree of the form $(r, T_1, T_2, \dots, T_i, S)$, where the $T_k$'s are empty trees with a single leaf, and the labelling corresponding to the chosen word $a_1 a_2 \cdots a_j$ is an initial section of the tree $S$ (in fact, in algebraic terms, what we did is modify the word for $y$ to get a suitable one yielding a unique tree after the empty trees). Note that by initial section we mean a prefix of the word obtained from the labelling of $S$ read in post-order, where we exclude the label of the root, i.e., if the root has a label, then the prefix is strict. We denote by $S_{n,i}$ the set of such Schröder trees, that is, those Schröder trees on $n$ leaves with $i+1$ child of the root, and such that the $i$ first child are leafs. Note that the tree that we attached to $x$ depends on a choice of word for $x$, but applying a defining relation in the word $x$ corresponds to applying a local move in the tree $S$, and this cannot make $S$ split into several trees since the root of $S$ is frozen (its label corresponds to the last letter of $y\neq 1$). Hence we can apply all local moves with all labels in the (strict) initial section corresponding to a word for $x$, and we keep a Schröder tree on $n-i+1$ leaves. In this way, forgetting the $i$ first empty trees, what we attached to $x$ is an equivalence class of a (strict) initial section of a Schröder tree on $n-i+1$ leaves under local moves, that is, a divisor of $\Delta_{n-i}$. This mapping is injective since one can recover a word for $x$ from the obtained Schröder tree on $n-i+1$ leaves easily by mapping $S$ to $(r, T_1, \dots, T_i, S)$, labelling such a tree, and reading the word obtained by reading the $i$ first empty trees and then the initial section.

It remains to show that it is surjective. Hence consider an initial section of a Schröder tree $S$ on $n-i$ leaves. We must show that, in the tree $(r, T_1, T_2, \dots, T_i, S)$, the initial section of $S$ is a word $a_1 a_2 \cdots a_j$ which labels an element $x'$ of $D_n^0$. Assume that $\rho_n$ is a left-divisor of $x'$. Then, using local moves only involving those labels in the initial section of $S$ corresponding to a word for $x'$, one can transform $(r, T_1, T_2, \dots, T_i, S)$ into a tree of the form $(r, T_1, T_2, \dots, T_i, T_{i+1}, S'_1, \dots, S'_e)$, i.e., $S$ can be split into several trees, the first one (corresponding to $\rho_n$) being an empty tree. This is a contradiction: to split $S$ into several trees, one would need to apply a local move involving the root of $S$, which is frozen since the initial section does not cover the root. Hence $x'\in D_n^0$, and our mapping is surjective. This completes the proof of the second point, as it is clear that our mappings preserve left-divisibility.     

For the first point, we have $D_n^{n+1}=\{\rho_n^{n+1}\}$, hence there is nothing to prove. To show that $D_n^0\cong \mathsf{Div}(\Delta_{n-1})$, one proceeds in a similar way as in the proof of point $1$. Let $x\in D_n^0$ and let $y$ such that $xy=\Delta_n$. Choose words for $x$ and $y$, and consider the corresponding Schröder tree $T=(r, T_1, \dots, T_k)$. Since $x\in D_n^0$, the initial section of $T$ corresponding to the word for $x$ must be a proper initial section of $T_1$. Using local moves on $T_2, \dots, T_k$ (which amounts to changing the word for $y$), we can find a Schröder tree that is equivalent to $T$ under local moves, and that is of the form $(r, \widetilde{T}_1, \widetilde{T}_2)$, where $\widetilde{T}_1$ still has the chosen word for $x$ as a proper initial section, and $\widetilde{T}_2$ is the empty tree with only one leaf. In particular $\widetilde{T}_1$ is a Schröder tree on $n$ leaf. Applying defining relations to words for $x$ amounts to applying local moves inside the first tree, and arguing as in the first point this establishes the isomorphism of posets between $D_n^0$ and $\mathsf{Div}(\Delta_{n-1})$.    
\end{proof}

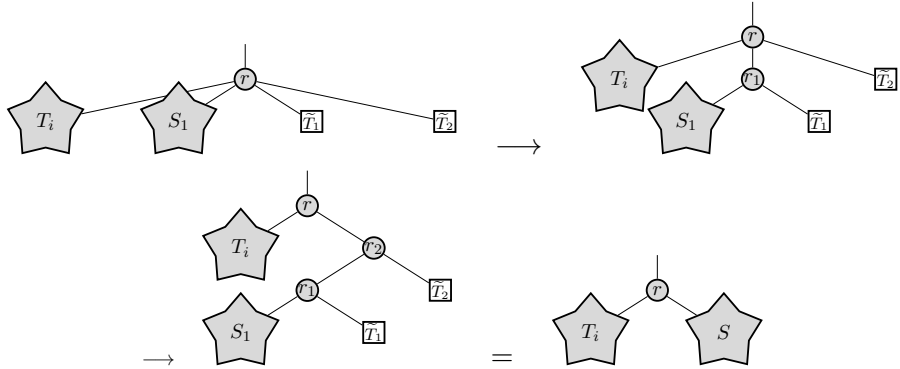
\begin{figure}[h!]
	\centering
	\scalebox{.7}{\begin{tikzpicture}
				[ level distance=15mm,
				level 1/.style={sibling distance=20mm,level distance=8mm},
				level 2/.style={sibling distance=25mm,level distance=8mm},
				inner/.style={circle,draw=black,fill=black!15,inner sep=0pt,minimum size=4mm,line width=1pt},
				leaf/.style={rectangle,draw=black,inner sep=0pt,minimum size=4mm,line width=1pt},  
				dots/.style={star,draw=black,inner sep=0pt,minimum size=8mm, text width=7mm,align=center, line width=1pt,fill=black!15},  
				root/.style={}               
				edge from parent/.style={draw,line width=1pt}]
				
				\node [root] {}
				child {node[inner] {$r$}
					child{node[dots] {$T_i$}}
					child{node[dots] {$S_1$}}
					child{node[leaf] {\footnotesize $\widetilde{T}_1$}}
					child{node[leaf] {\footnotesize $\widetilde{T}_2$}}
				};
		\end{tikzpicture}}
	\quad $\longrightarrow$
	\quad
	\scalebox{.7}{\begin{tikzpicture}
			[ level distance=15mm,
			level 1/.style={sibling distance=20mm,level distance=8mm},
			level 2/.style={sibling distance=25mm,level distance=8mm},
			inner/.style={circle,draw=black,fill=black!15,inner sep=0pt,minimum size=4mm,line width=1pt},
			leaf/.style={rectangle,draw=black,inner sep=0pt,minimum size=4mm,line width=1pt},  
			dots/.style={star,draw=black,inner sep=0pt,minimum size=8mm, text width=7mm,align=center, line width=1pt,fill=black!15},  
			root/.style={}               
			edge from parent/.style={draw,line width=1pt}]
			
			\node [root] {}
			child {node[inner] {$r$}
				child{node[dots] {$T_i$}}
						child{ node[inner] {$r_1$} 
					child{node[dots] {$S_1$}} 
				    child{node[leaf] {\footnotesize $\widetilde{T}_1$} }}
				child{node[leaf] {\footnotesize $\widetilde{T}_2$}}
			};
	\end{tikzpicture}}\\
\scalebox{.7}{$\longrightarrow$ \quad {\begin{tikzpicture}
			[ level distance=15mm,
			level 1/.style={sibling distance=20mm,level distance=8mm},
			level 2/.style={sibling distance=25mm,level distance=8mm},
			inner/.style={circle,draw=black,fill=black!15,inner sep=0pt,minimum size=4mm,line width=1pt},
			leaf/.style={rectangle,draw=black,inner sep=0pt,minimum size=4mm,line width=1pt},  
			dots/.style={star,draw=black,inner sep=0pt,minimum size=8mm, text width=7mm,align=center, line width=1pt,fill=black!15},  
			root/.style={}               
			edge from parent/.style={draw,line width=1pt}]
			
			\node [root] {}
			child {node[inner] {$r$}
				child{node[dots] {$T_i$}}
				child{node[inner] {$r_2$}
					child{ node[inner] {$r_1$} 
					child{node[dots] {$S_1$}} 
					child{node[leaf] {\footnotesize $\widetilde{T}_1$} }}
				child{node[leaf] {\footnotesize $\widetilde{T}_2$}}}
			};
\end{tikzpicture}}}
\quad $=$
\quad
\scalebox{.7}{\begin{tikzpicture}
		[ level distance=15mm,
		level 1/.style={sibling distance=20mm,level distance=8mm},
		level 2/.style={sibling distance=25mm,level distance=8mm},
		inner/.style={circle,draw=black,fill=black!15,inner sep=0pt,minimum size=4mm,line width=1pt},
		leaf/.style={rectangle,draw=black,inner sep=0pt,minimum size=4mm,line width=1pt},  
		dots/.style={star,draw=black,inner sep=0pt,minimum size=8mm, text width=7mm,align=center, line width=1pt,fill=black!15},  
		root/.style={}               
		edge from parent/.style={draw,line width=1pt}]
		
		\node [root] {}
		child {node[inner] {$r$}
			child{node[dots] {$T_i$}}
			child{node[dots] {$S$}}
		};
\end{tikzpicture}}
	\caption{Illustration for the proof of Proposition~\ref{dni_trees}.}\label{fig:merge}
\end{figure}

\section{Enumerative results}\label{sec_enumerative}

We have already seen (Theorem~\ref{thm_main_reduced}) that the words for $\rho_n^{n+1}$ are in bijection with Schröder trees on $n+1$ leaves. In this section, we give some additional enumerative results for several families of particular elements of $M_n$. 

\subsection{Number of simple elements}

\begin{cor}\label{cor_fib_even}
Let $n\geq 2$, and let $A_n:=|\mathsf{Div}(\Delta_n)|$. Then \begin{align}\label{fib_a} A_n=2 A_0 + 2 A_{n-1} + \sum_{i=1}^{n-2} A_i.\end{align} 
It follows that $A_n=F_{2n}$, where $F_0, F_1, F_2, \dots$ denotes the Fibonacci sequence $1, 2, 3, 5, 8, ...$ inductively defined by $F_0=1, F_1=2$, and $F_i=F_{i-1}+F_{i-2}$ for all $i\geq 2$.  The sequence of the $A_n$s is referred as \href{https://oeis.org/A001906}{A001906} in \cite{oeis}.
\end{cor}
\begin{proof}
The equality~\eqref{fib_a} follows immediately from the disjoint union $\Div(\Delta_n)=\coprod_{0\leq i \leq n+1} D_n^i$ and Proposition~\ref{dni_trees}. We have $A_0=F_0$, $A_1=3=F_2$, and it is elementary to check that the inductive formula given by~\ref{fib_a} is also satisfied by the sequence $F_{2n}$. This shows that $A_n=F_{2n}$ for all $n\geq 0$. 
\end{proof}

\begin{definition}
We call the lattice $(\mathsf{Div}(\Delta_n), \leq)$ the \defn{even Fibonacci lattice}. 
\end{definition}

\subsection{Number of left-divisors of the lcm of the atoms and odd Fibonacci lattice}

The set $\mathsf{Div}_L(\rho_n^n)$ of left-divisors of $\rho_n^n$ also forms a lattice under the restriction of left-divisibility, since it is an order ideal in the lattice $(\mathsf{Div}(\Delta_n), \leq)$. In terms of the Garside monoid $M_n$, the element $\rho_n^n$ is both the left- and right-lcm of the generators $\mathcal{S}=\{\rho_1, \rho_2, \dots, \rho_n\}$ (see~\cite[Corollary 4.17]{Gobet}). For $n\geq 1$ we set $B_n:=|\mathsf{Div}_L(\rho_n^n)|$. 

\begin{lemma}\label{lem_odd_fib}
We have $B_n= F_{2n-1}$ for all $n\geq 1$. The sequence of the $B_n$s is referred as \href{https://oeis.org/A001519}{A001519} in \cite{oeis}.
\end{lemma}
 
\begin{proof}
Let $x\in\mathsf{Div}(\Delta_n)$. We claim that $x\in\mathsf{Div}_L(\rho_n^n)$ if and only if $\rho_n x \in\mathrm{Div}(\Delta_n)$. Indeed, if $x\leq \rho_n^n$, there is $y\in M_n$ such that $xy=\rho_n^n$. We then have $\rho_n xy=\rho_n^{n+1}=\Delta_n$, hence $\rho_n x$ is a left-divisor of $\Delta_n$. Conversely, assume that $\rho_n x\in \mathsf{Div}(\Delta_n)$. It follows that there is $y\in \mathsf{Div}(\Delta_n)$ such that $\rho_n xy =\Delta_n=\rho_n^{n+1}$. By cancellativity we get that $xy=\rho_n^n$, hence $x\in \mathsf{Div}_L(\rho_n^n)$. 

It follows that $\mathsf{Div}_L(\rho_n^n)$ is in bijection with the set $$\{ \rho_nx \ \vert \ x\in \mathsf{Div}(\Delta_n)\} \cap \mathsf{Div}(\Delta_n).$$ But this set is nothing but $\coprod_{1\leq i \leq n+1} D_n^i$. It follows that $$B_n=|\mathsf{Div}(\Delta_n)|-|D_n^0|=|\mathsf{Div}(\Delta_n)|-|\mathsf{Div}(\Delta_{n-1)}|,$$ where the last equality follows from point~(1) of Proposition~\ref{dni_trees}. By Corollary~\ref{cor_fib_even} we thus get that $$B_n=A_n-A_{n-1}=F_{2n}-F_{2n-2}=F_{2n-1},$$ which concludes the proof.     
\end{proof}

\begin{definition}
We call the lattice $(\mathsf{Div}_L(\rho_n^n), \leq)$ the \defn{odd Fibonacci lattice}. 
\end{definition}

Both lattices for $M_3$ are depicted in Figure~\ref{simples_3}. 

\begin{rmq}
	Note that the set of right-divisors of $\rho_n^n$ also has cardinality $B_n$: in fact, the two posets $(\mathrm{Div}_L(\rho_n^n), \leq_L)$ and $(\mathrm{Div}_R(\rho_n^n), \leq_R)$ are anti-isomorphic via $x \mapsto \overline{x}$, where $\overline{x}$ is the element of $M_n$ such that $\overline{x}x=\rho_n^n$ (this element is unique by right-cancellativity). \end{rmq}

\subsection{Number of words for the divisors of the Garside element}
\begin{lemma}\label{lem:leftmostchild}
Let $T_1$ and $T_2$ be two Schröder trees with $n$ leaves labelled by $m\geq n-1$, and denote by $m_1$ and $m_2$ the corresponding words obtained by reading the labels in post-order. If the words $m_1$ and $m_2$ have a common prefix $x_1x_2\cdots x_l$, then $x_i$ labels a leftmost child in $T_1$ if and only if it labels a leftmost child in $T_2$. 
\end{lemma}
\begin{proof}
We prove the result by induction on the number of leaves. If $x_1\cdots x_l$ is obtained by reading all the vertices of $T_1=(r,S_1,\cdots, S_k)$, then $m_1 = x_1\cdots x_l=m_2$ and by Proposition \ref{prop:label_injectivity}, we have $T_1 = T_2$, hence there is nothing to prove. Otherwise, let $S_j$ be the first subtree of $T_1$ which is not covered by the word $x_1\cdots x_l$, similarly let $U_k$ the first subtree of $T_2=(r,U_1,\cdots U_v)$ which is not covered by $x_1\cdots x_l$. Looking at the proof of Proposition \ref{prop:label_injectivity}, we see that the first subtrees $S_1,\cdots, S_{j-1}$ are completely determined by the word $x_1\cdots x_l$, hence we have $j=k$ and $S_i = U_i$ for all $i< k$. Let $x_s$ be the letter of $x_1\cdots x_l$ labelling the first vertex of $S_j$. Let $m_1'$ be the subword of $m_1$ and $m_2'$ the subword of $m_2$ starting at the $x_s$. As explained in the proof of Proposition \ref{prop:label_injectivity}, we can determine the subword $m_1^{j}$ of $m_1'$ which correspond to $S_j$. The trees $U_j$ and $S_j$ do not need to have the same number of leaves. If one of the trees, say $U_j$, has less leaves, then one can apply local moves in the trees $U_{j},U_{j+1},\cdots U_v$ as in the proof of Proposition \ref{dni_trees} in order to obtain a tree $\tilde{U_j}$ with the same number of leaves as $S_j$. This will modify the word $m_2'$, but not the prefix $x_s\cdots x_l$, and $x_i$ labels a leftmost child in $U_j$ if and only if it labels a leftmost child in $\tilde{U}_j$ (see Figure \ref{fig:merge} for an illustration). After doing the modification, we consider the subword $m_2^{j}$ corresponding to the tree $\tilde{U}_j$ and apply the induction hypothesis to $m_1^{j}$ and $m_2^{j}$.\end{proof}

\begin{theorem}\label{thm_words_whole}
The set of words for the left-divisors of $\rho_{n}^{n+1}$ is in bijection with the set of Schröder trees with $n+2$ leaves. 
\end{theorem}
\begin{proof}
Let us denote by $s_k$ the number of Schröder trees with $k+1$ leaves, and $d_k$ the number of words for the divisors of $\rho_k^{k+1}$. 

Recall that $ \Div(\Delta_n)=\coprod_{0\leq i \leq n+1} D_n^i$, and let $d_{n}^{i}$ be the number of words for the elements of $D_n^{i}$. If $i = n+1$, then $\rho_n^{n+1}$ is the only element of $D_n^{i}$ and by Theorem \ref{thm_main_reduced}, there are $s_n$ words for this element, hence we have $d_{n}^{n+1} = s_n$.

Let $0\leq i \leq n$ and $w = x_1 \cdots x_l$ be a word for an element of $D_n^{i}$. The word $w$ is a strict prefix of a Schröder tree $T = (r,S_1,\cdots, S_k)$. By Lemma \ref{lemma_word}, $w = w_1 w_2$ where $w_1$ is a word for $\rho_n^{i}$ and $\rho_n$ is not a left divisor of $w_2$ (when $i=0$ the word $w_1$ is empty). Let $S_j$ be the last subtree of $T$ which has a vertex labelled by a letter of $w_1$. We can apply a succession of defining relations to $w_1$ in order to obtain $\rho_n^i$. These relations correspond to local move in the trees $S_1,\cdots,S_j$ which collapse all the trees $S_1,\cdots S_j$ to empty trees. In order to reduce $S_j$ to a list of empty trees we must use its root. Since the root is always the last label of the tree in post-order, the word $w_1$ covers all the first $j$ trees which have in total $i$ leaves. Since $\rho_n$ does not divide $w_2$, we see that $w_2$ is a (possibly empty) strict prefix of $S_{j+1}$. It is also possible to modify the trees $S_{j+2},\cdots, S_k$ without changing the first $j$ trees. Indeed, as in the proof of Proposition \ref{dni_trees} we can reduce the trees $S_{j+2},\cdots, S_{k}$ to empty trees and then merge them (until we can) to $S_{j+1}$. 

\begin{itemize}
\item[$\bullet$] When $i=0$,  after modification we obtain a tree $\tilde{T} = (r,\tilde{S},L)$ where $L$ the empty tree, $\tilde{S}$ is a tree with $n$ leaves and $w=w_2$ is a strict prefix of $\tilde{S}$. 
\item[$\bullet$] When $1\leq i \leq n$, we obtain a tree $\tilde{T} = (r,S_1,\cdots, S_j,\widetilde{S}_{j+1})$ and $w_2$ is a strict prefix of the tree $\widetilde{S}_{j+1}$ with $n+1-i$ leaves. 
\end{itemize}

In both cases, the tree $\widetilde{S}_{j+1}$ is obtained by possibly introducing new vertices to $S_{j+1}$, and as Figure~\ref{fig:merge} shows, these new vertices occur after the vertices of $S_{j+1}$, in post-order, hence $w_2$ is still a strict prefix of $\widetilde{S}_{j+1}$. Hence, we see that in the decomposition $w=w_1w_2$ of Lemma \ref{lemma_word}, the word $w_1$ is obtained by reading all the vertices of a Schröder tree with $i$ leaves and $w_2$ is a strict prefix of a Schröder tree, denoted by $\tilde{S}$, with $l_i +1$ leaves where $l_i= n-i$ leaves if $i\neq 0$ and $l_i = n-1$ if $i=0$.

Let $w = w_1 w_2$ be a word of an element of $D_{n}^{i}$ with $w_2$ having $t$ letters. Let $\tilde{S}$ be a Schröder tree with $l_i+1$ leaves having $w_2$ as a strict prefix. Then, we construct a word $\gamma(w)$ by first extracting $\tilde{S}$, then labelling it accordingly to its number of leaves (i.e., with $m=l_i$) and finally taking its first $t$ letters in post-order. Algebraically, it is easy to see how the word $\gamma(w)$ is obtained from $w_2$: if $w_i$ is the label of a leftmost child in $\tilde{S}$, we have $\gamma(w)_i = w_i$. Otherwise, since the tree $\tilde{S_j}$ has $l_i+1$ leaves, we have $\gamma(w)_i = w_i - n+ (l_i) $. A priori $\gamma(w)$ depends on the choice of a tree $\tilde{S}$, but Lemma \ref{lem:leftmostchild} tells us that $\gamma(w)$ only depends on $w_2$. The word $\gamma(w)$ is a prefix of a Schröder tree with $l_i+1$ leaves, hence it is a word for a divisor of $\Delta_{l_i}$. We have obtained a map $\gamma$ from the set of words for the elements of $D_n^{i}$ to the set of words for the divisors of $\Delta_{l_i}$. 

Conversely, if $z$ is a word of length $k$ for a divisor of $\Delta_{l_i}$, it is a prefix (strict since the root is not contributing) of a Schröder tree $S$ with $l_i+1$ leaves. We can view $S$ as a subtree of a Schröder tree with $n+1$ leaves by considering:

\begin{itemize}
\item $T = (r,S,L)$ when $i=0$;
\item $T = (r,\delta_{i},S)$ when $i\geq 1$.
\end{itemize}

Reading up to the first $k$ letters of the subtree $S$ produces a word $w=w_1w_2$ of an element of $D_{n}^{i}$ such that $\gamma(w) = z$. Hence $\gamma$ is surjective and we set $\epsilon(z)=w_2$. As before $\epsilon(z)$ only depends on $z$, not on the tree having $z$ as a prefix.

When $i=0$, the map $\gamma$ is injective, indeed if $w$ and $z$ are two words such that $\gamma(w) = \gamma(z)$, then by Lemma \ref{lem:leftmostchild} the labels of the leftmost child in $\gamma(w)$ and $\gamma(z)$ are the same, hence $w$ and $z$ are equal. This proves that $d_n^0 = d_{n-1}$. 

When $i\geq 1$, then $\gamma$ is far from being injective, since it forgets the first part of the tree. The set of words for the elements of $D_{n}^{i}$ is the disjoint union of two sets $E_1$ and $E_2$ where $E_1$ is the set of words $w=w_1w_2$ where $w_1$ covers exactly one tree $S_1$ and $E_2$ is the set of words where $w_1$ covers at least two trees. Note that when $i=1$, the set $E_2$ is empty otherwise both sets are non-empty. Indeed $E_2$ contains at least all the words of the form $\rho_n^i w_2$ and $E_1$ contains at least the words of the form $\rho_{i-1}\rho_n^{i-1}\rho_{n-i-1}w_2$ which correspond to the Schröder bush $\delta_i$ attached as the leftmost subtree of a Schröder tree.

If $z$ is a word for a divisor of $\Delta_{l_i}$, we compute the cardinality of the preimage of $z$ by $\gamma$ by looking at $\gamma^{-1}(z)\cap E_1$ and $\gamma^{-1}(z)\cap E_2$. If $i=1$, we obviously only consider the first case. The elements of $\gamma^{-1}(z)\cap E_1$ are obtained by concatenation of the word of a single Schröder tree with $i$ leaves and $\epsilon(z)$, and the elements of $\gamma^{-1}(z)\cap E_2$ are concatenation of the words of a forest with $i$ leaves made of at least two Schröder tree and $\epsilon(z)$.  Such a forest is nothing but a Schröder tree with $i$-leaves from which the root has been removed. So we have 

\[
 |\gamma^{-1}(z)\cap E_1| = s_{i-1} = |\gamma^{-1}(z)\cap E_2|.
\] 
Taking the sum on all possible words $z$, we have $d_n^{1} = s_0\cdot d_{n-1}$ and $d_n^{i}= 2\cdot s_{i-1} \cdot d_{n-i}$ when $n \geq i\geq 2$.

 We have obtained:
\[
d_n^{0} = d_{n-1};\quad d_n^{1} = s_0\cdot d_{n-1} = d_{n-1};
\]
and 
\[
d_n^{i}= 2\cdot s_{i-1} \cdot d_{n-i} \hbox{ when $n \geq i\geq 2$ and } d_n^{n+1} = s_{n}.
\]
By induction on the number of leaves, we have $d_i = s_{i+1}$, for every $i\leq n-1$, and
\begin{align*}
d_n &= 2s_n + 2\sum_{i=2}^{n}s_{i-1}s_{n-i+1}+s_n\\
& = 3s_n+2\sum_{i=1}^{n-1}s_{i}s_{n-i}.\\
\end{align*} 
Using generating functions, it is not difficult to check that this implies that $d_n = s_{n+1}$, see for example \cite[Theorem 5]{qi}. \end{proof}


\end{document}